\documentclass[11pt]{amsart}

\usepackage[left=1in,top=1in,right=1in,bottom=1in]{geometry} 

\input xy
\xyoption{all}
\usepackage{cite} 
\usepackage{fancyhdr}
\usepackage{mathrsfs}
\usepackage{amsmath} 
\usepackage{amssymb} 
\usepackage{amsfonts} 
\usepackage{multirow}
\usepackage{color}

\usepackage{hyperref}

\usepackage{tikz} 
\usepackage{wrapfig}
\newtheorem{theorem}{Theorem}
\newtheorem{lemma}[theorem]{Lemma} 
\newtheorem{proposition}[theorem]{Proposition} 
\newtheorem{corollary}[theorem]{Corollary}

\theoremstyle{definition}
\newtheorem{definition}[theorem]{Definition} 
\newtheorem{example}[theorem]{Example} 
\newtheorem{remark}[theorem]{Remark}

\newcommand\set[2]{\{#1 \ : \ #2\}}

\newcommand{\C}{\mathbb{C}}
\newcommand{\Cs}{{\C^{*}}}
\newcommand{\R}{\mathbb{R}}

\newcommand{\Z}{\mathbb{Z}}
\newcommand{\PP}{\mathbb{P}}
\newcommand{\Q}{\mathbb{Q}}
\renewcommand{\L}{\mathcal{L}}
\newcommand{\rL}{\L^{-1}}
\newcommand{\Circuits}{\mathcal{C}}
\newcommand{\sR}{\mathcal{R}}

\newcommand{\ol}{\overline}

\renewcommand\b{\mathbf{b}}
\newcommand\x{\mathbf{x}}

\newcommand\z{\mathbf{z}}

\renewcommand\1{\mathbf{1}}
\newcommand\initForm{\mathrm{in}_{-\1}}

\newcommand{\Stir}[2]{\genfrac{\{}{\}}{0pt}{1}{#1}{#2}}

\DeclareMathOperator{\diag}{diag}
\DeclareMathOperator{\codim}{codim}
\DeclareMathOperator{\supp}{supp}
\DeclareMathOperator{\rk}{rk}
\DeclareMathOperator{\sing}{Sing}
\DeclareMathOperator{\rowspan}{rowspan}

\author{Raman Sanyal}
\address{Fachbereich Mathematik und Informatik, Freie Universit\"at Berlin, Germany}
\email{sanyal@math.fu-berlin.de}
\author{Bernd Sturmfels}
\address{Department of Mathematics, University of California, Berkeley, USA}
\email{bernd@math.berkeley.edu}
\author{Cynthia Vinzant}
\address{Department of Mathematics, University of Michigan, Ann Arbor, USA}
\email{vinzant@umich.edu}

\title{The Entropic Discriminant}

\begin{document}

\begin{abstract}
The entropic discriminant is a non-negative polynomial associated to a matrix.
It arises in contexts ranging from statistics and linear programming to
singularity theory and algebraic geometry. It describes the complex branch
locus of the polar map of a real hyperplane arrangement, and it vanishes when
the equations defining the analytic center of a linear program have a complex
double root. We study the geometry of the entropic discriminant, and we
express its degree in terms of the characteristic polynomial of the underlying
matroid. Singularities of reciprocal linear spaces play a key role.  In the
corank-one case, the entropic discriminant admits a sum of squares
representation derived from the discriminant of a characteristic polynomial of
a symmetric matrix.
\end{abstract}

\maketitle

\section{Introduction}\label{sec:intro}

Entropy maximization for log-linear models in statistics leads to the
optimization problem
\begin{equation} \label{eq:opti1} 
    {\rm maximize} \,\,\, | x_1 x_2 \cdots x_n |  \,\,\, {\rm subject}
    \,\,{\rm to} \,\,\, A \x = \b . 
\end{equation}
Here $A$ is a fixed real $d\times n$-matrix of rank $d$ none of whose columns
are zero.  The right hand side vector $\b \in \R^d$ is a parameter that is
allowed to vary.  The problem (\ref{eq:opti1}) has a unique local solution in
the interior of each bounded region of the hyperplane arrangement
$\{x_i=0\}_{i\in[n]}$ inside the $(n-d)$-dimensional affine space $\set{ \x
\in \R^n}{A\x=\b }$.  The bounded regions are $(n-d)$-dimensional convex
polytopes. The number of bounded regions in this arrangement is constant 
for an open, dense set of vectors~$\b$. 
This number, $\mu(A)$, is a quantity known in matroid theory as
the {\em M\"obius invariant}.  The local
optima of (\ref{eq:opti1}) are the {\em analytic centers} of these $\mu(A)$
polytopes.  They are characterized~by
\begin{equation}
\label{eq:opti2}
A \cdot \x = \b \quad \hbox{and} \quad
 \left(\frac{1}{x_1}, \frac{1}{x_2}, \ldots, \frac{1}{x_n}\right) \,\,\,
\hbox{lies in the row space of $A$}.
\end{equation}
This translates into a system of polynomial equations in the
variables $x_1,\ldots,x_n$. It is known \cite{Sot, varchenko} that all complex
solutions of this system actually lie in $\R^n$.  Thus $\mu(A)$ is the algebraic
degree of~(\ref{eq:opti2}).

The aim of this article is to address the following question: Under what
condition on the right hand side $\b$ do two of the $\mu(A)$ solutions of
polynomial equations represented by (\ref{eq:opti2}) come together?  The set
of all complex right hand side vectors $\b \in \C^d$ for which this happens is
an algebraic variety $H_A$ in $\C^d$, called the {\em entropic discriminant}.
Under mild hypotheses on the matrix $A$, the entropic discriminant $H_A$ is a
hypersurface and we identify it with its defining polynomial, denoted
$H_A(\b)$.  This is a non-negative polynomial whose real 
zeros lie in certain linear subspaces of codimension~$2$.

\begin{example} 
\label{ex:dreifunf}
Let $d=3$ and $n=5$. The following $3 {\times} 5$-matrix  has M\"obius 
invariant $\mu(A) = 4$:
$$
A = \begin{pmatrix}
      1 & 0 & 0 & 1 & 1 \\
      0 & 1 & 0 & 1 & 0 \\
      0 & 0 & 1 & 0 & 1 \\
\end{pmatrix}
$$
The entropic discriminant of $A$ is a homogeneous polynomial
$H_A(b_1,b_2,b_3)$ of degree $8$.  It equals
$$
\begin{smaller}
\begin{matrix}
288 b_2^2b_3^2
(b_1^2 b_2^2
+  b_1^2 b_3^2 
+  b_2^2 s_1^2  
+ b_2^2 s_2^2
+ b_2^2s_3^2 
+  b_3^2 s_1^2 
+ b_3^2  s_2^2 
+  b_3^2 s_3^2 )
+ 1773 b_2^4 b_3^4 
+  720 b_2^2 b_3^2 (s_1^2 s_2^2 
+  b_1^2 s_3^2)  
\\
+ 192( b_1^2 b_2^4s_1^2  
+  b_2^4 s_2^2 s_3^2 
+  b_1^2  b_3^4 s_2^2
+  b_3^4 s_1^2 s_3^2) 
+  1216 (b_1^2 b_2^2 b_3^2 s_1^2 
 +  b_1^2 b_2^2  b_3^2s_2^2 
+ b_2^2 b_3^2s_1^2  s_3^2 
+ b_2^2 s_2^2 s_3^2 b_3^2) + 256 b_1^2 s_1^2 s_2^2 s_3^2\\
+ 320 (b_1^2b_2^2 s_1^2s_2^2 
+ b_1^2b_2^2s_1^2  s_3^2
+  b_1^2b_2^2 s_2^2 s_3^2
+ b_1^2b_3^2s_1^2 s_2^2 
+ b_1^2b_3^2  s_1^2 s_3^2
+  b_1^2b_3^2 s_2^2 s_3^2 
+b_2^2s_1^2  s_2^2 s_3^2
+ b_3^2 s_1^2 s_2^2 s_3^2 ),
\end{matrix}
\end{smaller}
$$
where $s_1=b_1-b_2$, $s_2=b_1-b_3$, and $s_3=b_1-b_2-b_3$. 
Thus $H_A(\b)$ is a sum of squares of quartics.

\noindent It coincides with the discriminant of the following system of equations in three unknowns:
$$
\begin{matrix}
1/z_1 + 1/(z_1+z_2) + 1/(z_1+z_3) & = & b_1 ,\\
1/z_2 \,+\, 1/(z_1+z_2) & = & b_2 ,\\
1/z_3 \,+\, 1/(z_1+z_3) & = & b_3. 
\end{matrix}
$$
These equations are equivalent to (\ref{eq:opti2}) if we take
$(z_1,z_2,z_3)$ to be coordinates for the row space of $A$.
There are four solutions for any $\b = (b_1,b_2,b_3) \in \C^3$.
They are distinct  if and only if $H_A(\b) \not= 0$. 
The entropic discriminant $H_A(\b)$ is a non-negative polynomial
having precisely four real zeros:
\begin{equation}
\label{eq:finding}
 V_\R(H_A) \quad = \quad \bigl\{
(0:1:0),  \,\, (0:0:1), \,\, (1:1:0),\,\,(1:0:1) \bigr\} \quad \subset \quad \PP^2. 
\end{equation}
The complex variety $V_\C(H_A)$ is a curve of degree $8$
in the projective plane with coordinates $(b_1{:}b_2{:}b_3)$.
That curve is singular at its four real points. In addition, it has $16$ isolated complex singularities.\hfill$\diamond$
\end{example}

We shall study the systems (\ref{eq:opti2}) for arbitrary $d$, $n$, and $A$.
The following is our main result:

\begin{theorem} \label{thm:intro}
    Let $A$ be a real $d \times n$-matrix of rank $d$ whose columns span 
     $\geq d+1$ distinct lines. The entropic discriminant is a hypersurface,
    defined by a homogeneous polynomial $H_A(\b)$ of degree
   \begin{equation} \label{eq:deg}
        \deg\, H_A(\b) \ = \ \, 2 (-1)^d \cdot (d \chi(0) + \chi'(0)),
    \end{equation}
    where $\chi(t)$ is the characteristic polynomial of the rank $d$ matroid
    of $A$. 
    For generic matrices $A$, this degree equals  $\,2 (n-d)
    \binom{n-1}{d-2}$.
    The polynomial $H_A(\b)$ is non-negative for all arguments in  $ \R^d$.  
\end{theorem}

The generic degree $2 (n-d) \binom{n-1}{d-2}$ is always an upper bound on the
degree of the entropic discriminant, and equality holds when the matroid of
$A$ is uniform; cf.~Proposition~\ref{prop:upperbound}. For example, for generic matrices $A$ of size $3 \times 5$,
the  degree of $H_A(\b)$ equals $16$, and not $8$ as in Example \ref{ex:dreifunf}.

This article is organized as follows.  In Section~\ref{sec:polar} we examine
the polar map of a product of linear forms.  The entropic discriminant is
shown to coincide with the branch locus of that polar map. For example,
consider the polar map of the binary form $f(z_1,z_2) =
z_1(z_1+2z_2)(z_1+3z_2)(z_1+az_2)$:
$$ \nabla_f \,:\, \PP^1 \,
\rightarrow
\, \PP^1 \, ,\,\,
(z_1:z_2) \,\mapsto \, \biggl(\frac{\partial f}{\partial z_1}(z_1,z_2):\frac{\partial f}{\partial z_2}(z_1,z_2)\biggl) .$$
The branch locus of this map consists of the four zeros of the binary quartic $H_A(b_1,b_2)$ in
Example~\ref{ex:zweivier} below.
This connects our study of $H_A(\b)$
to the topological theory of hyperplane arrangements
\cite{DP, Dol}, and to topics in classical algebraic geometry that are found
in Chapter 1 of Dolgachev's book~\cite{dolbook}.

Section~\ref{sec:n=d+1} is concerned with the important special case $n = d+1$.
Here the entropic discriminant has expected degree $d(d-1)$ and we
can write it explicitly as a sum of squares. This expression is derived
from known results on the discriminant of the characteristic polynomial
of a symmetric matrix \cite{bor46, ilyu92, lax98, newell}.
We then apply this to resolve two problems left open in the literature, namely 
the Sottile-Mukhin Conjecture \cite{AS} on the 
discriminant of the derivative of a univariate polynomial, and Conjecture 7.9 in \cite{stu02} 
concerning real critical double eigenvalues of a net of symmetric matrices.

For any linear subspace $\L$ of $ \C^n$, its {\em reciprocal} $\rL$ is defined as
the Zariski closure of the set 
\begin{equation}\label{eq:reciprocalplane}
  \biggl\{\, 
        \left( \tfrac{1}{u_1}, \tfrac{1}{u_2}, \ldots, \tfrac{1}{u_n}
    \right) \in \C^n \ : \ (u_1,u_2,\ldots,u_n) \in \L \cap (\C^*)^n \biggr\} . 
 \end{equation}
 In Section~\ref{sec:reciprocalplanes} we study the geometry of the {\em
 reciprocal plane} $\mathcal{L}^{-1}$, further extending the line of work
 from  Proudfoot-Speyer \cite{PS} to Huh-Katz~\cite{huh2}.  
 We identify a  minimal system of defining equations for $\mathcal{L}^{-1}$,
 we characterize the singular locus of $\mathcal{L}^{-1}$, and we determine all tangent cones.
 The relationship between that singular locus,
 the ramification locus of the map $A : \mathcal{L}^{-1} \rightarrow
 \PP^{d-1}$,  and the entropic discriminant $H_A(\b)$ is
 studied in detail  in Section~\ref{sec:hopefully}. 
In  Corollary \ref{cor:reallylast} we show
  that the real variety defined
 by the polynomial $H_A(\b)$ is a union of
 linear spaces of codimension $2$ in $\PP^{d-1}$.
 We saw this already for one instance in Example \ref{ex:dreifunf},
 where $d=3$ and the real variety is finite.

Theorem~\ref{thm:intro} is proved in Section~\ref{sec:ramification}. 
However, one subtle but essential point needs to be taken care of before that proof.
In order for  (\ref{eq:deg}) to be the correct degree, 
a more refined notion of entropic discriminant is required. 
Namely, we shall define $H_A({\bf b})$ as the polynomial defining
the cycle-theoretic branch locus of  the restriction to
 $\mathcal{L}^{-1}$ of the linear map $A: \C^n \rightarrow \C^d$, where
 $\mathcal{L}$ is the row space of $A$.   The following example justifies
 this ``fine print'' in Definition \ref{def:ED}.
 
\begin{example} 
\label{ex:zweivier}
Let $d = 2$, $n=4$ and  $A = \begin{pmatrix}
1 & 1 & 1 & 1 \\
0 & 2 & 3 & a \end{pmatrix} $ where $a$ is a real parameter.
For general values of $a$, the entropic discriminant is irreducible
and has degree $4$, as predicted by Theorem  \ref{thm:intro}:
$$
\begin{smaller}
\begin{matrix}
\!\! H_A(b_1,b_2) \quad = \quad
(2268 a^4-9720 a^3+11664 a^2) b_1^4
-(3000 a^4-12528 a^3+12960 a^2+5184 a) b_1^3 b_2   \qquad \\ 
+\,(1744 a^4-7980 a^3+10584 a^2-2160 a+5184) b_1^2 b_2^2 \\  \quad \qquad \qquad
-(500 a^4-2612 a^3+4680 a^2-3888 a+4320) b_1 b_2^3 
+(63 a^4-400 a^3+999 a^2-1350 a+1188) b_2^4.
\end{matrix}
\end{smaller}
$$
For special values of the parameter $a$, this expression factors over $\Q$.
For $a = 6$, it is the square
$\,972 (36 b_1^2 - 24 b_1 b_2 + 5 b_2^2)^2$. Thus, here the four
points of $V_\C(H_A) $ in $ \PP^1$ 
are two double points. \hfill$\diamond$
\end{example}

Our initial motivation for embarking on this project was a model in
theoretical neuroscience proposed by Hillar and Wibisono \cite{HilWib}. These authors
investigate the {\em retina equations} which characterize the maximum entropy
distribution for a graphical model $G$ with $n$ edges having continuous random
variables  on $d$ nodes that represent the firing pattern of $d$ neurons.
Their equations are
\begin{equation} \label{eq:sys1}
    \sum_{j \in \mathcal{N}(i)} \frac{1}{z_i + z_j} \,\, = \,\, b_{i} 
    \quad \hbox{for} \,\, i = 1,2,\ldots,d,
\end{equation} 
where $\mathcal{N}(i)$ is the set of all nodes that are adjacent to the node
$i$. The real numbers $b_1,b_2,\ldots,b_d$ are parameters that serve as the
sufficient statistics of the desired maximum entropy distribution.
 
To fit the system (\ref{eq:sys1}) into our framework, we introduce new
unknowns $x_{ij} = 1/(z_i+z_j)$ for all edges $\{i,j\} \in E(G)$.  This
translates  (\ref{eq:sys1}) into the linear
system $\, A \cdot \x \, = \, \b $, where $A$ is the
node-edge incidence matrix of $G$ and $\x= \bigl( x_{ij} :\{i,j\} \in E(G) \bigr)$ is a
column vector of unknowns. Of course, these unknowns obey the additional
constraints that $\x$ must lie in the reciprocal plane $\rL$, where
$\L$ is the row space of $A$.  Thus the retina equations of Hillar
and Wibisono fit our format (\ref{eq:opti2}):
\begin{equation} \label{eq:sys5}
    A  \cdot \x \, = \, \b 
    \quad \hbox{and} \quad \x \in \rL.
\end{equation}
The entropic discriminant $H_A(\b)$  characterizes measurements $\b$ for
which the retina equations (\ref{eq:sys1}) or  (\ref{eq:sys5}) have multiple
roots.  Of particular interest is the case $n = \binom{d}{2}$, when $G = K_d$
is the complete graph, and the sum in (\ref{eq:sys1}) is over $j \in \{1,\ldots,n\}\backslash \{i\}$.
The characteristic polynomial $\chi_d(t)$ of the corresponding matroid
was computed by Zaslavsky  \cite{zaslavsky82b}, in his work
 of colorings of signed graphs:
\begin{equation}
\label{eq:zaslavsky}
    \chi_d(t) \quad = \quad \sum_{k = 0}^d \bigl( \Stir{d}{k} +
    d\,\Stir{d-1}{k}\bigr)\,(t-1)^{(2)}_k.
\end{equation}
    Here $\Stir{d}{k}$ is the Stirling number of the second kind and
    $(x)^{(2)}_{k+1} = x(x-2)\cdots(x-2k)$ is the generalized falling
    factorial. One can also compute $\chi_d(t)$ with the exponential generating 
    function
\begin{equation}
\label{eq:stanley}
\sum_{d \geq 0} \chi_{d}(t) \cdot \frac{x^d}{d !}  \quad = \quad
(1+x) \cdot \bigl( 2 \cdot {\rm exp}(x) - 1 \bigr)^{(t-1)/2}, 
\end{equation}
found in \cite[Exercise 5.25]{Sta}.     
 Using these formulas, one obtains the
 first few values of the degree of $H_A({\bf b})$ and of the number of solutions of
 the retina equations on the complete graph $G = K_d$:
\begin{equation}
\label{eqn:stanley3}
\begin{matrix}
d \,\,\,\, & = \quad &  4 & 5 & 6 & 7 & 8 & 9 & 10 \\
{\rm deg}(H_A(\b))
 \,\, & = \quad &
 22 &
 270 &
 3148 &
 38990 &
 524858 &
 7705572 &
 123087958  \\
 \mu(A) & = \quad & 
7 &  51 & 431 & 4208 & 46824 & 586141 &  8161237
   \end{matrix} 
\end{equation}
The requisite combinatorics is developed in
Section~\ref{sec:matroids}. It covers material from matroid theory, focusing
on geometric interpretations of the characteristic polynomial and the M\"obius invariant. 
For instance, the third row in (\ref{eqn:stanley3}) is computed
from the series in (\ref{eq:stanley}) for $t = 0$, using formula (\ref{eq:unsigned}).


\section{The polar  map of a product of linear forms}
\label{sec:polar}

The $d \times n$-matrix $A = (a_{ij})$ determines a product of linear forms
in $d$ unknowns $\z = (z_1,\ldots,z_d)$:
\begin{equation}
\label{eq:polf}
 f(\z) \quad = \quad \prod_{j=1}^n \bigl( \,\sum_{i=1}^d a_{ij} z_i  \bigr) .
 \end{equation}
The hypersurface $V_\C(f)$ is an arrangement of $n$ hyperplanes
in  the complex projective space $\PP^{d-1}$. 
The {\em polar  map} of this hypersurface is the rational map
$$ \nabla_f \,:\,\PP^{d-1} \dashrightarrow \PP^{d-1} \,,\,\,\,
\z \,\mapsto \,  \left( \frac{\partial f}{\partial z_1}(\z):
\frac{\partial f}{\partial z_2}(\z): \cdots :
\frac{\partial f}{\partial z_d}(\z) \right) . $$
The base locus of
 $\nabla_f$ is the singular locus of $V_\C(f)$,
and this is the union of all  codimension-$2$ strata
in the hyperplane arrangement. If the columns of $A$ are
linearly independent then $\nabla_f$ is the
Cremona transformation of classical algebraic geometry,
and, in general, the polar  map $\nabla_f$ is also known as the
{\em polar Cremona transformation} \cite{Dol}.
The Jacobian of $\nabla_f$ is the Hessian of the polynomial $f$,
that is, the symmetric matrix of second derivatives.
We consider its determinant
$$ {\rm Hess} (f) \quad = \quad
{\rm det} \left(
\frac{\partial^2 f}{\partial z_i \partial z_j}
\right)_{\! 1 \leq i, j \leq d}.
$$
This is a homogeneous polynomial of degree $d(n-2)$.
Its zero set in $\PP^{d-1}$, denoted by $\,V_\C( {\rm Hess}(f))$,
is also referred to as the Hessian of $f$.
We are interested in the image of that hypersurface under~$\nabla_f$.

\begin{proposition} \label{prop:hess}
    The entropic discriminant equals the image of the Hessian under the
    polar  map:
    \begin{equation}
\label{eq:itf1}
     V_{\C}(H_A) \,\,\, = \,\,\, \hbox{closure of} \ \,\nabla_f \bigl( \,V_\C( {\rm
    Hess}(f)) \setminus V_\C(f) \, \bigr) . 
    \end{equation}
\end{proposition}

\begin{proof}
Let $\mathcal{L}^{-1}$ denote the reciprocal of the subspace $\mathcal{L}$
spanned by the rows of $A$, regarded as a subvariety of $\PP^{n-1}$. The
variety $\mathcal{L}^{-1}$ is the closure of the image of the map $\PP^{d-1} \dashrightarrow
\PP^{n-1}$ that takes a general point $\z = (z_1: \cdots : z_d)$ in
$\PP^{d-1}$ to $\, (\z A)^{-1} = \bigl( (\sum_{i=1}^d a_{i1} z_i)^{-1} :
\cdots: ( \sum_{i=1}^d a_{in} z_i)^{-1} \bigr)\,$ in $\PP^{n-1}$.  The
polar  map is the composition of this map with the linear projection
$\PP^{n-1} \dashrightarrow \PP^{d-1}, \,\x \mapsto A \x$.  In
symbols, we have $\, \nabla_f(\z) = A \bigl(  (\z A)^{-1} \bigr) $.
This observation shows that the fiber of $\nabla_f$ over a general real point
$\b \in {\rm Im}(\nabla_f)$ consists of $\mu(A)$ real points in $\PP^{d-1}$, namely,
the points represented by the analytic centers in the arrangement defined by
the coordinate hyperplanes in the affine space $\,\{ \x \in \R^n \,: \, A
\x = \b\}$.  This result was also obtained by Dimca and Papadima in
\cite[Corollary 4 (1)]{DP}.

For special complex points $\b \in \PP^{d-1}$, two of its $\mu(A)$ 
preimages under $\nabla_f$ may coincide.  At such a preimage $\z$ of
multiplicity $\geq 2$, the Jacobian of $\nabla_f$ drops rank, and 
the Hessian of $f$ vanishes at $\z$. Conversely, points $\z$ outside the
hyperplane arrangement $V_\C(f)$ at which the polynomial ${\rm Hess}(f)$
vanishes must be double roots of the system of equations $\nabla_f( \z) =
\b$.  Since the parametrization $\z \mapsto \x = (\z
A)^{-1}$ maps $\PP^{d-1}$ birationally onto the reciprocal plane
$\mathcal{L}^{-1}$, such double roots appear if and only the intersection
$\,\mathcal{L}^{-1} \,\cap \, \{\x \in \PP^{n-1}  : A \x= \b, \; x_1 x_2 \cdots x_n \neq0\}
\,$ has a point of multiplicity $\geq 2$. This condition on $\b$ is
the geometric definition of the entropic discriminant $H_A$.
\end{proof}

We have not yet addressed the question whether the entropic
discriminant actually has codimension $1$, and this may in fact not  be the
case.  For instance, if $A$ is the identity matrix and $f= z_1 z_2 \cdots z_d$
then $\nabla_f$ is the classical Cremona transformation on $\PP^{d-1}$ and $\,
{\rm Hess}(f) = (-1)^{d-1} (d-1) f^{d-2} $.  Here, the Hessian coincides with
the hyperplane arrangement, and the entropic discriminant is not a
hypersurface. We shall see that this is essentially the only exceptional~case.

The matrix $A = (a_{ij})$ is called {\em basic} if its column
rays lie on $d$ distinct lines in $\R^d$.  Since $A$ has
rank $d$ and no zero columns, this means that the distinct
column directions form a basis of $\R^d$. 

\begin{corollary} \label{cor:basic}
If $A$ is not basic then the entropic discriminant is a hypersurface in $\PP^{d-1}$.
\end{corollary}

\begin{proof}
A classical formula \cite[p.~660, Ex.~10]{muir}  for the Hessian determinant
of $f$ states that
\begin{equation} \label{eq:hessianSOS}
    \mathrm{Hess}(f) \quad = \quad  (-1)^{d-1} (n-1) f^{d-2} \cdot \sum_{I \in \binom{[n]}{d}} \!
    \det(A_I)^2 \prod_{k \in [n] \backslash I} (a_{1k} z_1 + a_{2k} z_2 + \cdots + a_{dk} z_d)^2
\end{equation}
where $A_I$ denotes the $d {\times} d$-submatrix of $A$ with column indices
$I$. 
If $A$ is not basic, then at least two summands are not scalar multiples of each other. 
This implies that the Hessian hypersurface is not contained in the hyperplane
arrangement $V_\C(f)$.  The polar  map $\nabla_f$ is a finite-to-one
morphism on the open set $\PP^{d-1} \backslash V_\C(f)$, and hence it
maps the Hessian to a hypersurface in $\PP^{d-1}$, namely $H_A$.
\end{proof}

\begin{corollary} \label{cor:nonnegative}
For any non-basic $A$,
the polynomial $H_A(\b)$ is homogeneous and nonnegative on~$\R^d$.
\end{corollary}

\begin{proof}
It suffices to prove this for the
square-free polynomial $\tilde H_A(\b)$ that vanishes
on $V_\C(H_A)$. Indeed, if $\tilde H_A(\b)$ is homogeneous
and nonnegative then so is any real product of its factors.

Homogeneity is straightforward since the
geometric definition ensures that
${\bf b} \in V_\C(H_A)$ implies $\lambda {\bf b} \in V_\C(H_A)$.
To show non-negativity,
let $K$ denote the subfield of $\R$ which is
generated by the entries of $A$.  We regard the entries of $\b =
(b_1,\ldots,b_d)$ as indeterminates over $K$. Let $L$ be the algebraic
closure of the rational function field $K(b_1,\ldots,b_d)$. Then the equation
$\,\nabla_f(\z) = \b \,$ has $\mu(A)$ distinct solutions with
coordinates in $L$. We substitute these solutions into the sum in
(\ref{eq:hessianSOS}) and we take their product in the field $L$.  The result
is a sum of squares in $L$ that is a symmetric polynomial in the roots.
It is invariant under the action of the Galois group of $L$ over
$K(b_1,\ldots,b_d)$ and thus lies in $K(b_1,\ldots,b_d)$. 
The sum of squares representation over $L$
ensures that this rational function is non-negative under all specializations of $\b$
to $\R$ at which it does not have a pole.  The numerator  of this rational function
is a product of the factors of $\tilde H_A({\bf b})$.
We conclude that  $\tilde H_A(\b)$ does not change signs on $\R^d$.
Hence, either $\tilde H_A({\bf b})$ or $-\tilde H_A({\bf b})$
 is non-negative on $\R^d$.
\end{proof}

The above argument shows that $H_A(\b)$ is non-negative
but it does not furnish a representation of $H_A(\b)$ as a sum
of squares of polynomials. We also note that the computation of
$H_A(\b)$ from ${\rm Hess}(f)$ is a 
task of elimination theory that is quite non-trivial
even for moderate values of $d$ and~$n$.

One case where the elimination problem can
be solved more easily is $d= 2$.
Here $f(z_1,z_2)$ is a binary form of degree $n$
enjoying the property that all its zeros on 
the line $\PP^1$ are defined over $\R$.
The polar  map $\nabla_f$ takes the complex projective line $\PP^1$ to itself.
This map has degree $n-1$, i.e.~the fiber over a general point
$ \b   \in \PP^1$ consists of $n-1$ points.
We are interested in those points $\b$ on the line $\PP^1$ for which two or more of the points
in its fiber collide. The Hessian of $f$ equals
$$ 
 {\rm det} \begin{pmatrix}
\frac{\partial^2 f}{\partial z_1^2} & 
\frac{\partial^2 f}{\partial z_1 \partial z_2} \\
\frac{\partial^2 f}{\partial z_1 \partial z_2}  &
\frac{\partial^2 f}{\partial z_2^2}
\end{pmatrix} \quad   = 
\quad (1-n) \cdot \!\! \sum_{1 \leq i < j \leq n} \!
(a_{1i} a_{2j} - a_{1j} a_{2i})^2 \, \cdot \!\! \prod_{k \in [n] \backslash \{i,j\}} 
(a_{1k} z_1 + a_{2k} z_2)^2.
$$
This is a binary form of degree $2n-4$, so it defines a configuration of
$2n-4$ points in $\PP^1$.  All points have non-real coordinates.
The entropic discriminant of $f$ 
is the image of these $2n-4$ points under the
polar  map $\nabla_f$. 
Proposition \ref{prop:hess} gives the following rule for computing
 the entropic discriminant:
 \begin{equation}
\label{eq:resultant}
 H_A(b_1,b_2) \quad = \quad {\rm Resultant}_{\z} \bigl( {\rm Hess}(f(\z)), 
b_2 \frac{\partial f}{\partial z_1}(\z)
- b_1 \frac{\partial f}{\partial z_2}(\z)
 \bigr)  .
\end{equation}
This formula can be rewritten as the discriminant of a binary form:
\begin{equation}
\label{eq:discriminant}
  H_A(b_1,b_2) \quad = \quad
\hbox{Discriminant}_{\z} \bigl( b_2 \frac{\partial f}{\partial z_1}(\z)
- b_1 \frac{\partial f}{\partial z_2}(\z) \bigr) .
\end{equation}

The binary form $H_A(b_1,b_2)$ has degree $2n-4$ provided no two columns of $A$ are parallel.  
Being nonnegative, the entropic discriminant  is a sum of squares of binary forms of degree  $n-2$ over $\R$.

\begin{example} 
Let $n=3$ and consider a general binary cubic with real zeros:
$$ f \,\,\, = \,\,\, (a_{11} z_1 + a_{21} z_2)  (a_{12} z_1 + a_{22} z_2)  (a_{13} z_1 + a_{23} z_2) . $$
The sum of squares representation in (\ref{eq:hessianSOS}) tells us that the Hessian of $f$ equals
$$(a_{11}a_{22}-a_{21}a_{12})^2(a_{13}z_1+a_{23}z_2)^2
+(a_{11}a_{23}-a_{21}a_{13})^2 (a_{12}z_1+a_{22}z_2)^2+(a_{12}a_{23}-
a_{22}a_{13})^2(a_{11}z_1 +a_{21} z_2)^2. $$
For any invertible matrix $U$, the entropic discriminant satisfies $H_{UA}(U\b)=H_A(\b)$.
This implies that $H_A(\b)$ can be written in the $2 \times 2$ minors $p_{ij}$ of the matrix
$2 \times 4$-matrix $(A , \b)$.  We have
\begin{equation}\label{eq:plucker}
H_A(\b) \;\;\; = \;\;\; 
(p_{12}\cdot p_{34})^2 \;+\;(p_{13}\cdot p_{24})^2 \;+ \; (p_{23}\cdot p_{14})^2.
\end{equation}
For $n=4$,  an expression for
$H_A(\b)$ in terms of the $2 \times 2$ minors the $2 \times 5$-matrix $(A , \b)$ is
\begin{equation}\label{eq:pluck2}
\begin{smaller}
\begin{matrix}
(p_{12}^2p_{34}p_{35}p_{45})^2+
(p_{13}^2p_{24}p_{25}p_{45})^2+
(p_{14}^2p_{23}p_{25}p_{35})^2+
(p_{14}p_{23}^2p_{15}p_{45})^2+
(p_{13}p_{24}^2p_{15}p_{35})^2\\
+(p_{12}p_{34}^2p_{15}p_{25})^2
+\frac{7}{2}(p_{23}p_{24}p_{34}p_{15}^2)^2+
\frac{7}{2}(p_{13}p_{14}p_{34}p_{25}^2)^2+
\frac{7}{2}(p_{12}p_{14}p_{24}p_{35}^2)^2+
\frac{7}{2}(p_{12}p_{13}p_{23}p_{45}^2)^2.
\end{matrix}
\end{smaller}
\end{equation}
At present we do not know how to extend the formulas
(\ref{eq:plucker}) and (\ref{eq:pluck2}) to $n \geq 5$.
\hfill $\diamond$
\end{example}

It is natural to ask how the formulas (\ref{eq:resultant}) and (\ref{eq:discriminant})
would generalize to $d \geq 3$, and the answer is given by the 
projective duality between the entropic discriminant and the
{\em Steinerian hypersurface} \cite[\S 1.1.6]{dolbook}.
If $f$ is any homogeneous polynomial of degree $n$ in $\z = (z_1,\ldots,z_d)$
then its Steinerian is\begin{equation}
\label{eq:Steinerian}
  {\rm St}_f(c_1,c_2,\ldots,c_d) \quad = \quad
\hbox{Discriminant}_{\z} \biggl(
 c_1 \frac{\partial f}{\partial z_1}(\z) + 
 c_2 \frac{\partial f}{\partial z_2}(\z) + \cdots + 
 c_d \frac{\partial f}{\partial z_d}(\z) \biggr).
\end{equation}
In this formula, we are taking the discriminant of a form of degree $n-1$,
namely, the polar of $f$ with respect to a generic point ${\bf c}$.
Corollary 1.2.2 in \cite{dolbook} tells us that
the hypersurface defined by ${\rm St}_f({\bf c})$ is
dual to the image of the hypersurface defined by  ${\rm Hess}(f(\z))$
under the polar map $\nabla_f$.

In our situation, the given form $f$ is a product of linear forms as in (\ref{eq:polf}), and 
some care needs to be taken in removing contributions from singularities. Indeed,
the Steinerian ${\rm St}_f$ of a hyperplane arrangement is supported
on that same hyperplane arrangement plus an extra component.
It is this extra component we are interested in. We call this hypersurface the
 {\em residual Steinerian} of $f$.
 
 \begin{corollary}
 The entropic discriminant of a $d \times n$-matrix $A$ is the hypersurface in $\PP^{d-1}$
 projectively dual to the residual Steinerian of the arrangement of $n$
 hyperplanes given by the columns of $A$.
 \end{corollary}

Let us briefly revisit the case $d=2$ from this point of view. We saw that
the entropic discriminant consists of $2n-4$ points on a projective line
with coordinates $(b_1:b_2)$. The Steinerian consists of $2n-4$
points on the dual projective line with coordinates $(c_1:c_2)$.
In our formulas (\ref{eq:resultant}) and (\ref{eq:discriminant}) we tacitly
identified these two lines and their point configurations via $(c_1:c_2)  =  (-b_2:b_1)$.

For $d \geq 3$, the formula (\ref{eq:Steinerian}) is less useful for the purpose
of computing $H_A(\b)$ because dualizing the residual Steinerian 
in a computer algebra system is hard. Instead, we find it preferable to use
\begin{equation}
\label{eq:itf2}
\langle H_A(\b) \rangle \,\,\, = \,\,\,
\bigl(\,
\langle \,{\rm Hess}(f(\z))\, \rangle \,+ \,
\langle \,\hbox{$2 {\times} 2$-minors of the $2 {\times} d$-matrix} \,( \,\b, \nabla_f ) \,\rangle \,\bigr)
: \langle\, \nabla_f \,\rangle^\infty.
\end{equation}
This ideal-theoretic reformulation of (\ref{eq:itf1}) is the direct generalization of (\ref{eq:resultant})
to $d \geq 3$.

Nevertheless, the (residual) Steinerian of a hyperplane arrangement remains a
beautiful topic in geometry, and its interplay with the combinatorics of the
entropic discriminant certainly deserves further study. We close this section
with an illustration of this for lines in the plane $\PP^2$.

\begin{example} 
This example was worked out with help from Igor Dolgachev.  Let $d = 3$ and
suppose the matroid of $A$ is uniform.  Thus $V_\C(f)$ is an arrangement of
$n$ lines in general position in $\PP^2$. By Theorem \ref{thm:intro}, the
entropic discriminant $H_A$ is a curve of degree $2(n-1)(n-3)$.  Its singular
locus consists of the $n$ columns of $A$.  By dualizing, we obtain the
Steinerian ${\rm St}_f$, a curve of degree $3(n-2)^2$. Each of the $n$ lines
occurs with multiplicity $n-2$ in the Steinerian. Removing these lines, we
find that the residual Steinerian $H_A^\vee$ is a curve of degree $3(n-2)^2 -
n(n-2) =  2 (n-2)(n-3)$.~$\diamond$
\end{example}


\section{The  codimension-$1$ case}\label{sec:n=d+1}

The discriminant of the characteristic polynomial of a symmetric matrix is
non-negative because real symmetric matrices have only real eigenvalues.  The
study of this discriminant is a classical subject in mathematics, going back
to an 1846 paper by Borchart \cite{bor46}. Explicit representations of this
discriminant as a  sum of squares were also presented in work of Newell
\cite{newell}, Ilyushechin \cite{ilyu92}, and Lax \cite{lax98}.  See \cite[\S
7.5]{stu02} for an exposition, and work of Domokos \cite{dom10} for the state
of the art.

In this section we establish a relationship between this subject and the
entropic discriminant.  We focus on the case $n = d+1$, and we express
$H_A(\b)$ as a specialization of the discriminant of the characteristic
polynomial of a symmetric matrix.  We shall use this to derive the following
result.

\begin{theorem} \label{thm:d+1}
    Let $A$ be a non-basic matrix with $d$ rows and $n = d+1$  columns. Then
    the entropic discriminant $H_A(\b)$ is a sum of squares of polynomials.
    Moreover, if the entries of $A$ are rational numbers then $H_A(\b)$ is a
    sum of squares in $\mathbb{Q}[b_1,\ldots,b_d]$.
\end{theorem}

\begin{example} \label{ex:dreivier} 
If $d = 3, n = 4$ and $A = \begin{smaller} \begin{pmatrix}
\,1 & 0 & 0 & -1 \,\\
\,0 & 1 & 0 & -1 \,\\
\,0 & 0 & 1 & -1  \,\end{pmatrix} \end{smaller}$ then $H_A(b_1,b_2,b_3)$
is the sum of $10$ squares
$$ \begin{matrix}
\frac{7}{4} b_1^4 (b_2{-}b_3)^2
+\frac{56}{27} (b_1{-}b_2)^2b_1^2b_2^2
+\frac{1}{108}(5b_1b_2{-}9b_1b_3{-}14b_2^2{+}18b_2b_3)^2b_1^2
+\frac{1}{27}(5b_1b_2 {-} 3b_1b_3 {-}8b_2^2 {+} 6b_2b_3)^2b_1^2 \\
+\frac{1}{9}(b_1b_2 {+}b_1b_3{-}2b_2b_3)^2(b_1-2b_2)^2 
+\frac{7}{108} (5b_1b_2{+}3b_1b_3{-}2b_2^2{-}6b_2b_3)^2b_1^2 
+\frac{1}{216}(13b_1b_2{-}21b_1b_3 \\ - 7b_2^2  -12b_2b_3+27b_3^2)^2b_1^2 
+\frac{1}{36}(5b_1^2b_2-7b_1^2b_3-7b_1b_2^2+4b_1b_2b_3+9b_1b_3^2+14b_2^2b_3-18b_2b_3^2)^2 \\
+\frac{1}{216}(5b_1b_2-21b_1b_3+b_2^2-12b_2b_3+27b_3^2)^2b_1^2 
+\frac{1}{36}(5b_1^2b_2-b_1^2b_3-4b_1b_2^2-8b_1b_2b_3+8b_2^2b_3)^2 .
\end{matrix}
$$
This expression is derived from the sum
of $10$ squares found at the top of page 97 in \cite{stu02}. \hfill
$\diamond$
\end{example}

\begin{proof}[Proof of Theorem \ref{thm:d+1}]
Let $A$ be a non-basic $d \times (d{+}1)$-matrix and let $v \in \R^{d+1}$ span
the kernel of $A$.  If $v$ has a zero coordinate, say $v_{d+1} = 0$, then
we can reduce our analysis to a smaller case, namely, a $(d{-}1)\times
d$-matrix obtained by taking the columns of $A$ modulo the last column. Hence
we may assume that all coordinates of $v$ are non-zero. 

Next, we claim that it suffices to prove our assertions for the special case where
\begin{equation} \label{eq:specialmatrix} 
    A \,\, = \,\,\begin{pmatrix}
1 & 0 & 0 &  \cdots & 0 & -1 \\
0 & 1 & 0 &   \cdots & 0 & -1 \\
0 & 0 & 1 &  \cdots & 0 & -1 \\
\vdots & \vdots &  \vdots  & \ddots & \vdots & \,\vdots \\
0 & 0 & 0 &  \cdots & 1 & -1 \\
\end{pmatrix}
\quad \hbox{and} \quad
v = (1,1,1,\ldots,1)^T. 
\end{equation}
That this suffices is ensured by the following transformation rule for the
entropic discriminant:
\begin{equation} \label{eq:transformation}
 H_{U A D}(\b) \,\, = \,\, H_A (U^{-1} \b) .
 \end{equation}
This identity holds for any invertible $d {\times} d$-matrix $U$ and any
invertible diagonal $n {\times} n$-matrix $D$, and its validity is easily seen
from the geometric definition of $H_A$.  We here use this for $n = d+1$.

We now fix $A$ and $v$ as in (\ref{eq:specialmatrix}).  Then $\mathcal{L} =
{\rm rowspace}(A)$ is the hyperplane $\, x_1 + x_2 + \cdots + x_n =
0$. Its reciprocal $\mathcal{L}^{-1}$ is the hypersurface of degree $d$ in
$\PP^{d}$ that is defined by the polynomial
\begin{equation} \label{eq:symdet1}
 \sum_{i=1}^{n} \prod_{j \not= i} x_j  \quad = \quad
\det
 \begin{pmatrix}
        x_1 + x_n \! & x_n       & \cdots & x_n\\
        x_n       & \! x_2 + x_n & \ddots & \vdots \\
      \vdots      & \ddots    & \ddots & x_n \\
        x_n       & \cdots    & x_n    &  x_{n-1} + x_n \\
    \end{pmatrix}.
\end{equation}
This symmetric determinantal representation of the $(n-1)$st elementary
symmetric polynomial is taken from~\cite{san11}.  The linear system $A \x = \b$
is equivalent to 
\begin{equation}
\label{eq:xbx}
    x_i \ = \ b_i + x_n \quad \text{ for } i = 1,2,\ldots,n-1.
\end{equation}
Thus the points satisfying (\ref{eq:opti2}) can be computed by 
substituting (\ref{eq:xbx}) into (\ref{eq:symdet1}) and equating the resulting univariate
polynomial to zero.  Setting $t = x_n$, the solutions to (\ref{eq:opti2})
correspond to zeros of 
\begin{equation}
\label{eq:symdet2}
p_{\b}(t) \,\, = \,\, {\rm det}\bigl(\, t E + {\rm diag}(b_1,b_2,\ldots,b_d) \,\bigr) 
\quad \hbox{where} \quad
E = \begin{pmatrix}
2 & 1 & 1 & \cdots & 1 \\
1 & 2 & 1 & \cdots & 1 \\
1 & 1 & 2 & \cdots & 1 \\
 \vdots & \vdots & \vdots & \ddots & \vdots \\
1 & 1 & 1 & \cdots & 2 
\end{pmatrix}.
\end{equation}
In particular, $\,H_A(\b)$ equals the discriminant of the univariate
polynomial $p_\b(t)$. The following proposition applied to $E$ and $X = -
\diag(b_1,\ldots,b_d)$ completes the proof of the theorem.
\end{proof}

\begin{proposition}\label{prop:rationalSOS}
    Let $E \in \R^{m \times m}$ be a symmetric positive definite matrix and $X$
    a symmetric matrix of indeterminates. Then the discriminant of the
    \emph{generalized} characteristic polynomial    $\,\det(t E-X) \,$ with
    respect to $t$ is a sum of squares in $\,\Q(E_{ij} :1 \le i,j \le m)[X_{ij} :1 \le i,j \le m]$.
\end{proposition}

\begin{proof}
\newcommand\hX{\hat{X}}
Since $E$ has a Cholesky factorization $E =
MM^{T}$, it follows that 
\begin{equation}\label{eq:genchar}
    \det(tE-X) \;\; = \;\;\det(E) \cdot \det(tI-M^{-1}X M^{-T}).
\end{equation}
We get a sum of squares formula from the known representations of the
discriminant of the characteristic polynomial of a real symmetric matrix.
However, our emphasis  lies
on the \emph{rationality} of the desired formula.
Following \cite[\S 7.5]{stu02}, let $\hX = M^{-1}X M^{-T}$ and consider the
linear map
\[
    \wedge_2 \R^m \,\rightarrow \,\mathrm{Sym}_2 \R^m \, ,\, \,\,
    Z\, \mapsto\, [\hX,Z] \ = \ \hX Z \,-\, Z\hX
\]
that takes a skew-symmetric matrix to the commutator with $\hX$.

Let $\{ W_{ij} = e_i \wedge e_j : 1 \le i < j \le m \}$ be the standard basis
for the space of skew-symmetric matrices and likewise $\{S_{ij} = e_i \cdot e_j : 1 \le i
\le j \le m \}$ the standard basis for the space of symmetric matrices. Let $\Phi$ be the
$\tbinom{m+1}{2} \times \tbinom{m}{2}$-matrix representing the linear map in
the chosen bases. By choosing suitable bases, it can be seen that the
eigenvalues of $\Phi^T \Phi$ are the squared pairwise differences of the
eigenvalues of $\hX$.  Hence the determinant of $\Phi^T\Phi$ is the
discriminant of $\det(t\,I-\hX)$. The sum of squares representation can be
obtained by applying the Binet-Cauchy theorem. 

To get a rational representation we apply the above reasoning to the slightly
altered map
\[
Z\, \mapsto \, [Z,X]_E \ := \ E^{-1}\,X\,Z \ - \ Z\,X\,E^{-1}.
\]
It is clear that a representation in the standard basis is over $\Q(E_{ij})$ and
hence yields an appropriate  sum of squares. To see that this actually yields
the discriminant for the generalized characteristic polynomial, choose bases
$\, W^\prime_{ij} \ = \ M^{-T}\,W_{ij}\,M^{-1} \,$ and $\,
S_{ij}^\prime \ = \ M^{-T}\,S_{ij}\,M^{-1}$ and verify
\[
    [X,W^{\prime}_{ij}]_E \ = \ M^{-T} \, [\hX,W_{ij}] \, M^{-1}.
\]
Hence, a representation in the new bases is given by $\Phi$ above.
\end{proof}

Evaluating the discriminant of $p_\b(t)$ in (\ref{eq:symdet2}) leads to
 the following data concerning the monomial expansion of the entropic
 discriminant $H_A(\b)$ of the particular matrix $A$ in  (\ref{eq:specialmatrix}):
\begin{equation}
\label{eq:somedata}
\begin{matrix}
d &  2 & 3 & 4 & 5 & 6  \\
\hbox{degree of $H_A(\b)$} &
2 & 6 & 12  & 20 & 30 \\
\hbox{number of monomials} & 3 & 19 & 201 & 3081 & 62683  \\
\hbox{leading (lex) monomial} &
\,\,b_1^2 \,\,& \,b_1^4 b_2^2 \,& b_1^6 b_2^4 b_3^2  & b_1^8 b_2^6 b_3^4 b_4^2 & 
b_1^{10} b_2^8 b_3^6 b_4^4 b_5^2 \\
\end{matrix}
\end{equation}

\begin{remark} 
The entropic discriminant $H_A(\b)$ is a symmetric 
polynomial in $b_1,\ldots,b_d$ since the set of rows of the matrix $A$ is invariant under
permutations. It thus admits a unique representation as a polynomial
in the elementary symmetric polynomials
$$ \quad e_k \,\,\,\,\,=  \sum_{1 \leq i_1 < \cdots < i_k \leq n} \!\!\! b_{i_1} b_{i_2} \cdots b_{i_k}
\qquad \quad \hbox{for} \,\,\, k=1,2,\ldots,d. $$
For example, if $d=4$ then the $201$ terms
in $b_1,b_2,b_3,b_4$ translate into only $16$ terms in 
$e_1,e_2,e_3,e_4$:
$$
\begin{smaller}
\begin{matrix}
H_A(\b) \quad =  \quad
432 e_1^4 e_4^2-432 e_1^3 e_2 e_3 e_4+128 e_1^3 e_3^3+108 e_1^2 e_2^3 e_4-36 e_1^2 e_2^2 e_3^2-
2160 e_1^2 e_2 e_4^2 +1800 e_1 e_2^2 e_3 e_4 \quad \\ \quad
+120 e_1^2 e_3^2 e_4
-540 e_1 e_2 e_3^3-405 e_2^4 e_4+
135 e_2^3 e_3^2+2400 e_1 e_3 e_4^2+1800 e_2^2 e_4^2-2700 e_2 e_3^2 e_4+675 e_3^4-2000e_4^3.
\end{matrix}
\end{smaller}
$$
\end{remark}

As an application of our theory, we are now able
to answer two questions from the literature. The first deals with the discriminant of the derivative of a 
univariate polynomial. According to Alexandersson and Shapiro
\cite[Theorem 1.4]{AS}, Frank Sottile and Eugene Mukhin formulated this conjecture  
at the AIM meeting ``Algebraic systems with only real solutions'' in October 2010. 

\begin{corollary}\label{cor:univariate} 
The discriminant of the derivative of a univariate polynomial $f(t)$ of degree 
$n$ is a sum of squares of polynomials in the differences of the roots of $f(t)$.
\end{corollary}

\begin{proof} 
Let $\mathcal{D}_n = {\rm discr}_t \bigl(f'(t)\bigr)$.
We shall write $\mathcal{D}_n$ as a specialization of the entropic discriminant and use the sum of squares decomposition given in Theorem~\ref{thm:d+1}. 
Consider the univariate polynomial $f(t)=\prod_{i=1}^n(t-a_i)$. 
Notice that $x_i=t-a_i$ provides a parametrization for the one-dimensional affine space
$\{A\x = \b\}$, where we take $A$ as in \eqref{eq:specialmatrix} and $b_i=a_n-a_i$ for $i=1, \hdots, n-1$. 
We plug this parametrization into the polynomial \eqref{eq:symdet1} that defines $\L^{-1}$.
This yields the derivative $f'(t) = \sum_{j=1}^n\prod_{i\neq j}(t-a_i)$. 
Thus $f'(t)$ equals the polynomial  $p_{\b}(t)$ of  \eqref{eq:symdet2}
whose discriminant (with respect to $t$) equals $H_A(\b)$. We conclude that $\mathcal{D}_n$ equals
 the entropic discriminant $H_A((a_n-a_i)_{i\in[n-1]})$.
 Using Theorem~\ref{thm:d+1}, 
 we conclude that $\mathcal{D}_n$ is a sum of squares in $\Q[\b]=\Q[(a_n-a_i)\;:\;i\in[n-1]]$.
\end{proof}

Our techniques can also be applied to answer a question that was left 
open in \cite[\S 7.5]{stu02}.  Namely, we
conclude this section by proving Conjecture~7.9 of \cite{stu02}.

\begin{corollary}\label{cor:eigvalues}
There exist three real symmetric $d\times d$-matrices $C_0, C_1, C_2$ such
that all $\binom{d+1}{3}$ pairs of complex numbers $(x,y)$ for which
$C_0+xC_1+yC_2$ has a critical double eigenvalue are real.
\end{corollary}

\begin{proof}
Consider the symmetric matrix $\hat X= M^{-1} X M^{-T}$ with $X = {\rm diag}(b_1,\ldots,b_d)$
in the proof of Proposition \ref{prop:rationalSOS}. Its entries are linear forms
in $\b = (b_1, b_2, \hdots, b_d)$. We replace the unknowns $b_i$
  by generic real affine-linear forms  in two variables $x$ and $y$, say
 $\,b_i = w_i + u_i  x + v_i  y\,$ for $i = 1,\hdots ,d$. 
 The symmetric matrix resulting from this substitution is a
 net of real symmetric $d \times d$-matrices:
\[ \hat X \quad = \quad
  C_0 + x\,C_1 + y\,C_2.
\]
The real values of $(x,y)$ for which this matrix has
a critical double eigenvalue corresponds to the intersections of this affine
plane with the real variety of $H_A(\b)$, with $A$ given in (\ref{eq:specialmatrix})

We claim that the real radical of the entropic
discriminant is the codimension-$2$ ideal
\begin{equation} \label{eq:d+1choose2}
 \sqrt[\R]{\langle H_A(\b)\rangle} \,\quad = \quad
\bigcap_{1 \leq i < j \leq d} \!\! \langle b_i, b_j \rangle \quad \cap \,\,
\bigcap_{1 \leq i < j < k \leq d} \!\!\! \langle b_i- b_j, b_j-b_k \rangle . 
\end{equation}
This identity follows from the geometric description of $H_A(\b)$ in terms of
colliding analytic centers.
 Indeed, the hyperplane arrangement defined by
$\{x_i = 0\}$ in $\{\x : A \x = \b\}$ consists of $n=d+1$ points on a line.
They form $d$ bounded segments. The analytic centers of two segments collide
if and only if three of the $d+1$ points coincide.  There are such
$\binom{d+1}{3}$ triples, each imposing a condition of codimension $2$.  They
are expressed by the $\binom{d+1}{3}$ prime ideals in the intersection
(\ref{eq:d+1choose2}).

Since the real variety of $H_A(\b)$ is a union of $\binom{d+1}{3}$ real linear spaces, 
each of its intersection points with the plane 
$\b={\bf w} + {\bf u}x + {\bf v} y$ is also defined over the reals.  
Therefore all $\binom{d+1}{3}$ symmetric
matrices with a critical double eigenvalue in the net $C_0 + xC_1 + yC_2$ 
have real entries.
\end{proof}

\begin{remark} 
In general, the issue of determining $ \sqrt[\R]{\langle H_A(\b)\rangle} $ is  very subtle.
The validity of the identity (\ref{eq:d+1choose2}) above rests formally on 
the prime decomposition of the real radical ideal $ \sqrt[\R]{\langle H_A(\b)\rangle} $
described in Corollary \ref{cor:reallylast}. See also
Example  \ref{ex:directA} for the particular matrix $A$ in
(\ref{eq:specialmatrix}). 
\end{remark}

\section{Matroids and Graphs}
\label{sec:matroids}

In this section we discuss the notions from matroid theory which are needed
for the statement and proof of  Theorem \ref{thm:intro}.  We also discuss
various matroids arising from graphs, including those representing the
Hillar-Wibisono model (\ref{eq:sys1}).  Matroid theory is a classical subject
in combinatorics with many (axiomatic) paths leading to it. For us, matroids
come in the form of matrices and hence we take the concrete approach via
realizable matroids. For more on this subject see~\cite{oxley, Sta}.

Our given matrix $A = (A_1, A_2,\dots, A_n) \in \R^{d \times n}$ is identified
with an ordered collection of $n$ vectors that span $d$-space.  The
corresponding matroid $M = M(A)$ records all linear dependencies among these
vectors.  A subset $I \subseteq [n] = \{ 1, 2, \dots, n\}$ is called {\em
(in)dependent} whenever $A_I = (A_i \;:\; i \in I )$ is linearly
(in)dependent.  The {\em rank} $\rk(I)$ of $I$  is the rank of $A_I$. The rank
of the matroid $M$ is the rank of $A$.  A \emph{circuit} is an
inclusion-minimal dependent subset, and $I$ is independent if it does not
contain a circuit.  A subset $F \subseteq [n]$ is a \emph{flat} if $F$ equals
$\{ i \in [n] : A_i \in \mathrm{span} (A_F) \}$, that is, if $F$ precisely
indexes a collection of vectors contained in some linear subspace.
Equivalently, $F$ is a flat if and only if it meets every circuit in $\ge 2$
elements or not at all.  Flats will play an important role in
Sections~\ref{sec:reciprocalplanes} and~\ref{sec:ramification}. The
\emph{lattice of flats} $L(M)$ is the collection of all flats, ordered by
inclusion, with minimal element $\hat{0} = \{ i : A_i = 0 \}$ and maximal
element $\hat{1} = [n]$.  The lattice of flats represents combinatorial
information about the containment relations of the various subspaces spanned
by subsets of columns of $A$. It is one of the central objects in the
enumerative theory of matroids. 

A different but equivalent perspective on $M(A)$ and $L(M)$ is by means of the
hyperplane arrangements alluded to in Section~\ref{sec:polar}.  The $n$
columns of $A$ are normal to $n$ linear (not necessarily distinct) hyperplanes
$h_1,h_2,\dots,h_n \subseteq \R^d$.  In this context, a subset $I \subseteq
[n]$ is independent if and only if the intersection of $\{ h_i : i \in I\}$
has codimension $|I|$. The collection of linear subspaces obtained by
intersections of these hyperplanes is isomorphic to the lattice of flats
$L(M)$ when partially ordered by \emph{reverse} inclusion. For a generic
vector $\b$, the matroid associated to $(A,\b)$ is called the \emph{free
extension} of $M(A)$.  The hyperplane arrangement corresponding to the free
extension is obtained by adding a hyperplane such that intersections with
flats are transverse.

At the beginning of Section~\ref{sec:intro}, we considered a different
arrangement of affine hyperplanes associated to $A$. To relate this to $M(A)$,
observe that the $n$ coordinate hyperplanes $\{x_i = 0\}$ in $\R^n$ induce a
hyperplane arrangement in $\ker(A) \cong \R^{n-d}$. This arrangement
corresponds to the \emph{dual} matroid to $A$, namely $M(B)$, where $B$ is an
$(n-d)\times n$ matrix whose rows form a basis for $\ker(A)$. The hyperplane
arrangement $\{g_1,g_2,\dots,g_n\}$ in $\R^{n-d}$ associated to the columns of
$B$ is linearly isomorphic to the arrangement of the $n$ coordinate
hyperplanes in $\ker A$. Dually, the hyperplane arrangement $\{ h_1, h_2,
\dots, h_n \}$ given by the columns of $A$ yields a linearly isomorphic
representation of the arrangement of coordinate hyperplanes inside $\ker(B)$.

The matroid dual to the free extension by $\b$ is called the \emph{free
co-extension}, which corresponds to the linear arrangement of the $n+1$
coordinate hyperplanes in $\ker ((A,\b))$. Here we distinguish the last
hyperplane $g_\infty$ as the hyperplane ``at infinity''. Restricting the
arrangement to $g_\infty$ recovers the original arrangement in $\ker (A)$.
The arrangement that will be central to our cause, is the arrangement of $n$
affine hyperplanes given by the intersection of coordinate hyperplanes in $\{
\x : A\x = \b\}$, for generic $\b$. This is the restriction of the $g_i$ to
some parallel displacement $g_\infty + t$ (for some generic $t \not\in g_\infty$).
Alternatively, this is the affine arrangement $\{ \hat g_i = g_i + t_i \subset
\R^{n-d} : i = 1,\dots,n \}$ where the displacements $t_i \in \R^{n-d}$ are
generic.  Thus, the arrangement of coordinate hyperplanes in $\{ A\x = \b\}$
can be obtained by a generic, parallel perturbation of the hyperplanes
$g_1,g_2,\dots,g_n$.

Associated to $L = L(M)$ is its \emph{M\"obius function} $\mu_L : L \times L
\rightarrow \Z$, which is defined by $\mu_L(F,F) = 1$,
\[    
    \mu_L(F,H) \,\,\, = \,\,\,\, - \! \sum_{F \subseteq G \subset H} \mu_L(F,G) 
\]
if $F \subseteq H$, and $\mu_L(F,H)=0$ otherwise.  The {\em characteristic
polynomial} of $M$ is defined by
\[
    \chi_M(t) \ = \ \sum_{F \in L(M)} \mu_{L(M)}(\hat{0},F)\,t^{rk(M) - rk(F)}.
\]
The \emph{(unsigned) M\"obius invariant} of $M$, or of the matrix $A$, is the positive integer
\begin{equation}
\label{eq:unsigned}
    \mu(A) \,\,=\,\, \mu(M(A)) \,\,=\,\, |\mu_L(\hat 0, \hat 1)| \,\,=\,\, (-1)^d\chi_M(0). 
\end{equation}
Here the last equality comes from Rota's Sign Theorem. 

Evaluations of the characteristic polynomial have nice combinatorial
interpretations in terms of hyperplane arrangements \cite{GreZas,Sta}.  The
M\"obius invariant $\mu(A)$ equals the number of bounded regions of the
restriction of the $n$ coordinate hyperplanes to $\{ A\x = \b\}$, for generic
$\b$. This fact played an important role in \cite[\S 3]{DSV}. The proof is a
straightforward deletion-contraction argument, using that $\mu(A)$ and the
number of bounded regions in $\{A\x = \b\}$ adhere to the same recurrence
relations.  
This number is related to the \emph{beta invariant} of the free extension
$(A,\b)$, 
\[
\beta(A,\b) \ := \ (-1)^{\rk(A,\b)} \sum_{I \subseteq [n+1]} (-1)^{|I|}
\rk(A,\b)_I \ = \ (-1)^{\rk(A)} \sum_{F \in L} \mu_L(\hat{0},F)
\ = \ \mu(A)
\]
where the middle equality is taken from~\cite[Prop.~7.3.1]{white}.  The
geometric content of this statement was proved by Greene and
Zaslavsky~\cite[Eqn.~3.1]{GreZas} and, in a more algebro-geometric context,
in~\cite{DP}. The beta invariant is unchanged under duality of matroids and
thus $\beta(A,\b) = \mu(A)$ is the number of bounded regions for the
coordinate arrangement in $\{ A\x = \b \}$ when $\b$ is generic.  

The proof of the following observation illustrates the typical line of
arguments in matroid theory.

\begin{proposition} \label{rmk:MuOne}
    The M\"obius invariant $\mu(A)$ equals $1$ if and only if the matrix $A$
    is basic (defined in~Corollary \ref{cor:basic}) if and only if its
    geometric lattice $L(M)$ is the Boolean lattice of all 
    subsets of~$[d]$.
\end{proposition}

\begin{proof}
    An equivalent statement appears in \cite[Corollary 4 (2) (b1)]{DP}.  For
    completeness, we here include a combinatorial proof.  By the definition of
    the lattice of flats, we can assume that $A$ has no zero columns and that
    no two columns are proportional. Hence, the matrix $A$ is basic if and
    only if $M(A)$ is isomorphic to the uniform matroid $U_{d,d}$ whose
    lattice of flats is the Boolean lattice of all subsets of $[d$]. The
    M\"obius invariant of the matroid $U_{d,d}$ is $\mu(A) = 1$; see
    Example~\ref{ex:uniform} below.  Conversely, if $A$ is non-basic, there is
    a column $e \in [n]$ that is not an \emph{isthmus}, that is, not contained
    in every basis. For such an element $e$, the M\"obius invariant satisfies
    the deletion-contraction identity
    \[
        \mu(A) \ = \ \mu(A \backslash e) \ +  \ \mu(A / e).
    \]
    By Rota's Sign Theorem, the M\"obius invariant is always a positive integer and
    hence $\mu(A) \ge 2$.
    \end{proof}

We have now defined the combinatorial ingredients for
the degree (\ref{eq:deg})  of the entropic discriminant.
With this in place, we derive the value of that degree
for generic matrices $A$ stated in Theorem \ref{thm:intro}:

\begin{example}[Uniform matroids] 
\label{ex:uniform} 
A generic $d \times n$-matrix $A$ with $d \le n$ gives rise to the \emph{uniform
matroid} $M = U_{d,n}$ in which every subset of cardinality $\le d$ is
independent. The corresponding lattice of flats is a truncated Boolean lattice
in which a subset $F \subseteq [n]$ is a flat if and only if $|F| < d$ or $F =
[n]$. The M\"obius function on the Boolean lattice is $\mu(F,G) =
(-1)^{|G\setminus F|}$ for $F \subseteq G$. Hence
\[
\chi_{U_{d,n}}(t) \ = \ t^d - n t^{d-1} \ + \ \cdots \ + \ (-1)^{d-1}
\tbinom{n}{d-1}\, t \ + \  (-1)^d\,\tbinom{n-1}{d-1}.
\]
Note that $t=1$ is always a zero of the characteristic polynomial.
The number of solutions of the equations (\ref{eq:opti2}) for generic $A$ equals
$\,\mu(A) = \binom{n-1}{d-1}$. The degree (\ref{eq:deg}) of
the entropic discriminant equals
$$ \quad 
2 (-1)^d \cdot (d \chi_{U_{d,n}}(0) + \chi_{U_{d,n}}'(0))  \,\, = \,\,
2 \left[ d \binom{n-1}{d-1} - \binom{n}{d-1} \right] \,\, = \,\,
2 (n-d) \binom{n-1}{d-2}.  $$
As we will see in Proposition~\ref{prop:upperbound} below,
this quantity is an upper bound for 
 fixed $n$ and $d$.
\hfill$\diamond$
\end{example}

\emph{Graphical matroids} are an important class of examples.  Let $G$ be a
graph on $d$ nodes with $n$ edges and $c$ connected
components. For an arbitrary but fixed orientation of the edges, let $A_G$ be
the $d \times n$ \emph{incidence matrix} of node-edge pairs, with entries $+1,
-1, 0$ if the node is in-coming, out-going, or non-incident for the edge.
Reorienting an edge of $G$ results in scaling the corresponding column of
$A_G$ by $-1$ and hence leaves the matroid $M_G = M(A_G)$ invariant. Note that
$A_G$ has rank $d-c$ and a matrix representation of full rank can be obtained
by selecting a node in every connected component of $G$ and deleting the
corresponding rows.  The matroid concepts above have natural interpretations
in graph-theoretic terms: circuits correspond to cycles and independent sets
to forests. The characteristic polynomial $\chi_G(t) = \chi_{M_G}(t)$ in this
context is also called the \emph{tension polynomial} and $t^c\chi_{G}(t)$
counts the number of proper $t$-colorings of $G$ where $t \in \Z_+$.  
Returning to the setting of Section~\ref{sec:polar}, the hyperplane arrangement given by the columns of 
$A_G$ is the graphic arrangement associated with $G$, which has the defining polynomial
\[
    f_G(\z) \quad \,\,\,  = \quad \prod_{(i,j) \in E(G)} \!\! (z_i - z_j). 
\]
The entropic discriminant  $H_G(\b)$ is the equation of the branch locus
of the gradient map $\nabla_{f_G}$. As $A_G$ does not have full rank, we
assume $z_i = 0$ for the rows $i$ that were deleted
when passing from $A_G$ to a rank $d-c$ matrix with $d-c$ rows.
The gradient map $\nabla_{f_G}$ is discussed in
\cite[Remark 8]{Huh}.

\begin{example}[Cycles] 
Let $G = C_{d+1}$ be the cycle with $n = d{+}1$ edges. Every collection of $d$
or fewer edges is independent and $M_G$ has a unique circuit.  The
truncated matrix $A_{C_{d+1}}$ has corank $1$ and $M_G$ is the uniform matroid
$U_{d,d+1}$. The reciprocal plane $\rL_{A_G}$ is a hypersurface of degree $d$,
and the entropic discriminant
$H_{C_{d+1}}(\b)$ is the polynomial of degree $d(d-1)$ seen in Section~\ref{sec:n=d+1}.
 $\diamond$
\end{example}

\begin{example}[Complete graphs]  
    As the name says, the complete graph $G = K_{d+1}$ has all possible edges
    on $d+1$ nodes. The characteristic polynomial is the chromatic polynomial divided by~$t$:
\[
    \chi_{K_{d+1}}(t) \,\,=\,\,  (t-1) (t-2) \cdots (t-d).
\]
The reciprocal plane $\mathcal{L}_{K_{d+1}}^{-1}$  is
a projective variety of degree $\,(-1)^d\chi_{K_{d+1}}(0) \,=\, d!\,$. We find that
\[ \deg\,H_{K_{d+1}}(\b) \,\,\, = \,\,\, 2 \cdot
    \left( d  - 1 - \frac{1}{2}  - \frac{1}{3} - \frac{1}{4} \cdots -
    \frac{1}{d} \, \right) \cdot d\, !
\]
is the value of the matroid invariant (\ref{eq:deg}) for the
incidence matrix $A_{K_n}$ of the complete graph $K_n$.

For example, for $d=3$ we get the complete graph on $4$ nodes, with
node-edge incidence matrix
\[
A_{K_4} \quad = \quad  \left(\begin{array}{@{}rrrrrr@{\,}}
     1 & 1& 1& 0 & 0 &0   \\
    -1 &  0& 0 & 1& 1& 0  \\
     0  &-1&  0 &-1& 0 & 1 \\
       0& 0 &-1& 0 &-1&-1 \\
   \end{array}\right).
    \]
The reciprocal plane $\mathcal{L}_{K_4}^{-1}$ is a 
surface of degree $6$ in $\PP^5$. Its homogeneous prime ideal is generated by 
four quadrics, one for each of the $3$-cycles in $K_4$.
The entropic discriminant $H_{A_{K_4}}$ defines a curve
in the projective plane $\PP^2$. That curve has
 degree $14$ and it has precisely six real points.
\hfill$\diamond$
\end{example}

The matroids associated with the retina equations (\ref{eq:sys1}) are
different from the matroids $M_G$ above.  Their matroids correspond to
\emph{all-negative graphs} in Zaslavsky's theory of signed
graphs~\cite{zaslavsky82}. Here, an all-negative graph $-G$ is an ordinary
graph with all edges marked by $-1$. The incidence matrix $A_{-G}$ of  $-G$
has entries in $\{0, 1\}$ where a $1$ signifies an incident node-edge pair.
The corresponding matroid $M(-G) = M(A_{-G})$ is the \emph{unoriented cycle
matroid}. The matroid-theoretic notions for $M(-G)$ translate to (signed)
graph concepts but the transitions are more involved. For all-negative graphs,
the circuits correspond to  \emph{even primitive walks}, that is, even cycles
or pairs of odd cycles connected by a simple path (of length possibly $0$);
cf.~\cite[Cor.~7D.3(e)]{zaslavsky82}.
For the state of the art on algebraic properties of the circuits of $A_{-G}$
see the recent work of Tatakis and Thoma~\cite{TT}. Evaluations of the
characteristic polynomial have interpretations in terms of \emph{signed
colorings}~\cite{zaslavsky82b}. 

 For example, the all-negative complete graph $-K_4$ on four
nodes has the incidence matrix
\begin{equation}
\label{eq:-K_4}
A_{-K_4} \quad = \quad 
\begin{pmatrix}
 \,1 \, & \, 1 \, & \, 1 \, & \, 0 \, & \, 0 \, & \, 0 \, \\
 \,1 \, & \, 0 \, & \, 0 \, & \, 1 \, & \, 1 \, & \, 0 \, \\
 \,0 \, & \, 1 \, & \, 0 \, & \, 1 \, & \, 0 \, & \, 1 \, \\
 \,0 \, & \, 0 \, & \, 1 \, & \, 0 \, & \, 1 \, & \, 1 \, \\
 \end{pmatrix}.
\end{equation}
Note that this matrix has rank $4$. Its matroid has
the characteristic polynomial
\[
    \chi_{-K_4}(t) \ = \ t^4 - 6\,t^3 +15\,t^2 -17\,t+7.
\]
The characteristic polynomials for the all-negative complete graphs
on any number of nodes were 
computed by Zaslavsky~\cite[Eqn.~5.8]{zaslavsky82b}. We 
presented his formula in the introduction in (\ref{eq:zaslavsky}).
An equivalent formula in terms of generating functions
due to Stanley \cite[Ex.~5.25]{Sta} was shown in~(\ref{eq:stanley}).

For the matrix (\ref{eq:-K_4}),  the reciprocal variety  $\rL_{-K_4}$ is defined by the three cubic equations
\[
\begin{matrix}
x_{12} x_{13} x_{24} - x_{12} x_{13} x_{34} -x_{12} x_{24} x_{34} +x_{13} x_{24} x_{34}
& = & 0 ,\\
     x_{13} x_{14} x_{23} - x_{13} x_{14} x_{24} - x_{13} x_{23} x_{24} +x_{14} x_{23} x_{24}
     & = & 0 , \\
     x_{12} x_{14} x_{23} -x_{12} x_{14} x_{34} -x_{12} x_{23} x_{34} +x_{14} x_{23} x_{34}
     & = & 0.
     \end{matrix}
\]
The task in (\ref{eq:sys1}) is to solve these cubic equations together with linear equations $\,A_{-K_4} \cdot \x = \b $
for the six unknowns $x_{12}, \ldots, x_{34}$. The number of solutions to this
system is $\mu(M_{-K_4}) =  7$, and all seven solutions are real
when the $b_i$ are real.
One of the solutions has only positive coordinates if and only if
the column vector $(b_1,b_2,b_3,b_4)$ of parameters lies in the
convex polyhedral cone spanned by the columns of~$A_{-K_4}$.
The entropic discriminant $H_{-K_4}(b_1,b_2,b_3,b_4)$ characterizes
parameter values for which the number of solutions is less than $7$.
It is a surface in $\PP^3$ of degree $2 (4 \cdot 7 - 17) = 22$.
The M\"obius invariant $\mu(M_{-K_d})$ and the degree of $H_{-K_d}$
for larger values of $d$ are displayed in (\ref{eqn:stanley3}).

We close this section with the remark that  
the study of characteristic polynomials of matroids
is an active area of current research in combinatorics.
The coefficients of $\chi(t)$ have
interpretations as face numbers of broken circuit
complexes and form a log-concave sequence.
This log-concavity was a longstanding conjecture recently resolved by Huh \cite{Huh} for
graphs and in its full generality by Huh-Katz~\cite{huh2}.  Their methods of
proof are based on the  geometry of reciprocal planes, our topic
in the next section. Specifically, a key player in \cite{huh2} is the
tropicalization of the graph of $\mathcal{L} \dashrightarrow \mathcal{L}^{-1}$.


\section{Geometry of Reciprocal Planes}\label{sec:reciprocalplanes}

Entropic discriminants arise as branch loci
from projecting reciprocal planes. This was already
hinted at in the proof of  Proposition \ref{prop:hess}.
We shall make this precise in Section \ref{sec:ramification}, where it will
be our main ingredient in the proof of Theorem \ref{thm:intro}.
In this section we build up to this proof by deriving some results
on reciprocal planes. We believe that these results are
of interest in their own right. 

We fix a $d \times n$-matrix $A$ of rank $d$ with no zero columns. Its rows span 
a $(d-1)$-dimensional subspace $\L$ in the projective space $\PP^{n-1}$.
Let $T$ denote the dense torus in $\PP^{n-1}$, i.e.~the
complement of the $n$ coordinate hyperplanes $\{x_i = 0 \}$.   The
\emph{reciprocal plane} $\rL$ is the Zariski closure of the coordinate-wise
inverse of $\L\cap T$, as in \eqref{eq:reciprocalplane}.  It is an
irreducible projective variety of dimension $d-1$. The inversion map from
$\mathcal{L}$ to $\mathcal{L}^{-1}$ is birational and it is an isomorphism
on $\L \cap T$. The coordinate ring of the reciprocal plane $\C[\x]/I(\rL)$ is 
isomorphic to the Orlik-Terao algebra, studied in \cite{SchT}.

Proudfoot and Speyer~\cite{PS} showed that $\rL$ is stratified by the flats of
the matroid $M(A)$. Recall that $J \subseteq [n]$ is a flat of $M(A)$ if and
only if $\rk(A_J) < \rk(A_{J^\prime})$ for all $J^\prime \supsetneq J$. Here
$A_J$ denotes the column-induced submatrix of $A$. For a flat $J \subseteq
[n]$, the corresponding stratum $\rL \cap \PP^J = \{ p \in \rL : \supp(p)
\subseteq J \}$ is isomorphic to $\rL_J$, the reciprocal plane associated to
the restriction $A_J$.  We shall investigate these boundary strata and the
singular locus $\sing(\rL)$ of $\rL$. 

We can identify each circuit $C$ of the matroid $M(A)$ with a vector $v\in \R^n$ in the kernel of $A$ 
with support ${\rm supp}(v)=C$. 
Let  $\Circuits(A) \subseteq \R^n$ denote the set of representative vectors for all circuits of $M(A)$. 
To each $v \in
\Circuits(A)$ we associate a polynomial
\begin{equation}\label{eq:circuitpoly}
h_v(\x) \ \,  = \ \sum_{i \in {\rm supp}(v)} v_i \prod_{j \not= i} x_j \ \, = \,\ \x^{{\rm supp}(v)} \!\!
\sum_{i \in {\rm supp}(v)} \frac{v_i}{x_i}.
\end{equation}
These {\em circuit polynomials} cut out the variety $\rL$.
 In fact, Proudfoot and Speyer proved
the much stronger result that $\{ h_v : v \in \Circuits(A) \}$ is
a \emph{universal Gr\"obner basis} for the prime ideal of $\rL$.

 As the set of all circuits is
typically rather large, one might be interested in a smaller set of polynomials 
to cut out $\L^{-1}$. The following characterizes subsets of the set of
circuit polynomials that cut out $\rL$ set-theoretically.
As we saw above, the boundary of $\L^{-1} \backslash T$ in $\L^{-1}$ is described by flats of 
$M(A)$. Recall that $J \subseteq [n]$ is a flat if and only if $|J^c \cap {\rm supp}(v)| \not= 1$ for every circuit $v \in
\Circuits(A)$. We say that a non-flat $J\subset[n]$ is \emph{exposed} by a circuit $\,v \in
\Circuits(A)\,$ if \,$|J^c \cap {\rm supp}(v)| = 1$.

\begin{proposition}\label{thm:gen_set}
    Let $\mathcal{B} \subseteq \Circuits(A)$ be a subset of the set of
    circuits. 
    The corresponding set of circuit polynomials
    $\{ h_v \;:\; v \in \mathcal{B} \}$ cuts out $\rL$ set-theoretically if
    and only if $\mathcal{B}$ exposes every non-flat.
\end{proposition}

\begin{proof}
Suppose that $J$ is a non-flat that is not exposed by any $v \in \mathcal{B}$.
Then, for each $ v \in \mathcal{B}$, 
either $|J^c \cap {\rm supp}(v)| \geq 2$, in which case $h_v$ is identically zero on $\PP^J$,
or $J^c \cap {\rm supp}(v) = \emptyset$, in which case $v$ is a circuit of $A_J$ and  $h_v$
vanishes on $\mathcal{L}_{J}^{-1}$. This shows that the subvariety of $\PP^{n-1}$ cut out by
 $\{ h_v \;:\; v \in \mathcal{B} \}$ contains $\mathcal{L}_{J}^{-1}$
and hence is strictly larger than $\mathcal{L}^{-1}$.

Conversely, assume that $\mathcal{B}$ exposes every non-flat.
Let $p $ be any zero of   $\{ h_v \;:\; v \in \mathcal{B} \}$,
and let $J = \supp(p)$. Suppose that $J$ is a non-flat of $M(A)$.
Then there exists $v \in \mathcal{B}$ that exposes $J$.
This means that exactly one of the terms of $h_v$ is non-zero at $p$, 
and hence $h_v(p) \not = 0$. We conclude that $J$ is a flat of $M(A)$.
Since $\mathcal{L}_{J}^{-1}$ is a boundary stratum     of $\rL$, it
    is sufficient to prove that $p_J^{-1} \in \rowspan A_J$. 
        For this, we shall prove that
    the kernel of $A_J$ is spanned by
        $\{ v_J : v \in \mathcal{B} \,\, \hbox{and} \,\, {\rm supp}(v) \subseteq J \}$. 
    Let $J_0 \subseteq J$ be a basis of $M(A_J)$. If $J_0$ is not a
    flat, then there is a circuit $v_1 \in \mathcal{B}$ supported on $J$ such
    that $\supp(v_1) \backslash J_0 = \{ j_1 \}$. Set $J_1 = J_0 \cup \{j_1\}$
    and repeat the procedure. This process terminates after $k = |J| -
    \rk(A_J) = \dim \ker A_J$  many steps. The matrix of the resulting
        circuits $v_1,\ldots,v_k$ is lower-triangular 
     and hence gives a basis for       $\ker( A_J)$. 
\end{proof}

This previous result highlights
the connection of our study to tropical geometry.

\begin{remark} \label{rem:yuyu} 
    Combining Proposition \ref{thm:gen_set} with the results of \cite{yy}, we
    infer that a collection of circuits cuts out the reciprocal plane
    $\L^{-1}$ set-theoretically  if and only if it constitutes a
    \emph{tropical basis} for the tropicalization of the linear space $\L$.
    Yu and Yuster~\cite[Sect.~2.2]{yy} showed that different inclusion-minimal
    tropical bases for $\L$ need not have the same cardinality. Specifically,
    the uniform matroid $U_{2,5}$ has  inclusion-minimal tropical bases of
    size $5$ and $6$. From this we can infer that Proposition~\ref{thm:gen_set} 
    holds only set-theoretically and  not in the
    ideal-theoretic or scheme-theoretic sense. 
\end{remark}

\begin{example}
\label{ex:uniformreciprocal} 
If the matroid $M(A)$ is uniform, then
the prime ideal of the reciprocal plane $\mathcal{L}^{-1}$ is minimally
generated by $\binom{n-1}{d}$ polynomials of degree $d$.
This can be seen as follows.
The initial ideal of $\mathcal{L}^{-1}$ with respect to the
reverse lexicographic term order is generated by the square-free
monomials representing broken circuits.
These are $x_{i_1} x_{i_2} \cdots x_{i_d}$
where $1 \leq i_1 <  \cdots < i_d \leq n-1$, so their number is
$\binom{n-1}{d}$. By \cite[Lemma 5]{yy}, the basic circuits obtained by
adding the last element $n$ form an inclusion-minimal tropical
basis for $\mathcal{L}^{-1}$. Hence, by Remark \ref{rem:yuyu}, the corresponding $h_v$ minimally cut out
 $\mathcal{L}^{-1}$. It follows that they form a minimal generating set for the 
ideal of $\mathcal{L}^{-1}$.
 \hfill$\diamond$
\end{example}

We now come to the main result in this section, namely, the characterization
of the tangent cone of the reciprocal plane $\rL$ at any point.  For the sake
of convenience, we here identify the $(d-1)$-dimensional projective variety
$\rL$ with the corresponding $d$-dimensional affine variety in~$\C^n$.

The \emph{tangent cone} $\mathrm{TC}_p X$ of
a variety $X \subset \C^n$ at a point $p$ is a scheme 
that describes the local behavior of $X$ around $p$.
For a polynomial $f \in \C[\x]$,
the \emph{initial form} $\initForm(f)$ is the non-zero homogeneous
component of $f$ of minimal degree. The tangent cone $\mathrm{TC}_p X$ is
defined by the ideal
\begin{equation}\label{eq:tangentcone}
    I(\mathrm{TC}_p X) \ = \ \langle  \initForm(f(\x + p)) \;:\; f \in
    I(X) \rangle.
\end{equation}
The following result shows that   the tangent cone of
$\rL$ at any point is reduced and irreducible. Here we use $\L_{A/J}$ to denote the 
$(d-{\rm rank}(J))$-dimensional linear space $\L/\L_J$ in $\C^n/\L_J \simeq \C^{n-|J|}$.

\begin{theorem} \label{thm:tangent_cone}
    Let $A \in \R^{d \times n}$ be a matrix of full row rank $d$ and let $\rL$
    be its reciprocal plane in $\C^n$. For any point $p \in \rL$ with support $J$,
    the tangent cone is isomorphic to the direct product
    \begin{equation}\label{eq:tangentcone2}
        \mathrm{TC}_p\rL \ \cong \ \L_J \ \times  \rL_{A/J},
    \end{equation}
    where ``$\,\cong$'' denotes the equality of affine schemes after a linear transformation
    in $\C^n$.
\end{theorem}

\begin{proof}
We inspect the initial forms of the
    circuit polynomials that define $\rL$.  Let $v \in \Circuits(A)$ be a circuit
    with support $C = \supp(v)$ and circuit polynomial $h_v(\x)$ as in 
    \eqref{eq:circuitpoly}.
    First suppose that $C \not\subset J$.
     We write $v = v' + v''$ where
    $\supp(v') = C\cap J$ and $\supp(v'')=C\backslash J$. Then
    $v''$ is a circuit of the matroid $M(A/J)$ obtained from $M(A)$ by
    contraction at $J$. The following identity holds:
    \[
        h_v(\x + p) \,\,\, = \,\,\, \x^{C\setminus J} \cdot h_{v'} (\x + p) \ + \ (\x + p)^{C\cap J} \cdot        h_{v''}(\x).
    \]
    Every term of $\x^{C\setminus J} h_{v'} (\x + p)$ has degree at least $|C \backslash J|$
    while 
    \[ \initForm \bigl( (\x + p)^{C\cap J} h_{v''}(\x) \bigr) \,\,\, = \,\,\, p^{C\cap J} \cdot        h_{v''}(\x) \]
    has degree $|C \backslash J|-1$. 
    This means that
    $h_{v''}({\bf x})$ is the initial form of $ h_v(\x + p)$.
   As every circuit $w$ of the contraction $M(A/J)$ is the restriction $v''$ of some
    circuit $v$ of $M(A)$, we conclude that the tangent cone ideal at $p$ contains 
    the prime ideal     $\langle h_w(\x) : w \in \Circuits(A/J) \rangle$
    that defines $ \rL_{A/J}$.
    
    Next suppose that $C \subseteq J$. Then $p$ is a regular point on the
    hypersurface $\{h_v = 0\}$, and the initial
        form $\initForm(h_v(\x+p))$ is the differential $D_p h_v$. 
    The differential of $h_v$ at the point $p$ is
    \[
    D_ph_v(\x)  \ \ = \;\; \sum_{i=1}^n \frac{\partial h_v}{\partial x_i}(p) \; x_i \;\; =  \ \ p^C \sum_{i \in C} \frac{x_i}{p_i}\; \Bigl(\;\sum_{j \in
    C \setminus i} \frac{v_j}{p_j} \;\Bigr) \ \  = \ \ 
    - \sum_{i \in C} \frac{v_i}{p_i^2} x_i.
    \]
    The second equality holds because $p^{-1}$ lies in $\rowspan(A_J) \cap (\Cs)^J$, 
    and the third equality follows from the fact that $-v_i/p_i =
    \sum_{j \not= i} v_j/p_j$, since $v$ is a circuit for $A_J$. Thus, $D_ph_v$
    vanishes on the rowspan of $A_J \diag(p_J)^2$, denoted $\L_{J}(p)$,
    and all circuits vanishing on this row span arise this way.
    
    We have shown that the prime ideal of the irreducible variety 
    $\L_{J}(p)  \times \rL_{A/J}$ is contained in the ideal of the tangent cone
    of $\rL$ at $p$. Since both ideals have the same height, 
    and the former is prime, it follows that they are equal.
    This proves the equality of schemes that was claimed.
    \end{proof}

A closer inspection of the proof reveals that the initial forms of $h_v(\x
+p)$ for $v \in \Circuits(A)$ furnish a universal Gr\"obner basis for the
tangent cone   of $\rL$ at $p$. In particular, we  obtain a simple description of the
tangent space of $\rL$ at a point $p$ by taking those initial forms
that are linear. 

\begin{corollary}\label{cor:cotangent}
    For a point $p \in \rL$ with support $J$, 
    the tangent space is  orthogonal to the space
     spanned by the circuits of the $d\times |J|$-matrix $A_J \diag(p_J)^2$ 
    and the circuits of $A/J$ of size $2$.
\end{corollary}

\begin{proof}
    The tangent space is cut out by the linear forms in the ideal of
    the tangent cone. From the initial forms in the proof of
    Theorem~\ref{thm:tangent_cone}, we see that $\initForm(h_v)$ is
    linear whenever $|{\rm supp}(v) \cap J^c| \le 2$.
    If $C \subseteq J$, then $\initForm(h_v)$  corresponds to a circuit of $A_J
    \diag(p_J)^2$. Otherwise, the two elements of $C \backslash J$ are
        parallel in the contraction $A/J$ and the corresponding circuit polynomial is linear.
  \end{proof}

This is closely related to \cite[Thm.~2.3]{SchT}, which investigates the quadratic 
component of the ideal $I(\rL)$.
Our discussion shows that the dimension of the tangent space is constant on each stratum of $\rL$.
We obtain the following characterization of the singular locus of 
the reciprocal plane $\rL$. 

\begin{corollary}\label{thm:singLA}
    The singular locus of the reciprocal plane $\rL$ is pure of
    codimension $2$. It is the union of all boundary strata
    $\rL_J$ such that the contraction $M(A/J)$ is a non-basic matroid.
\end{corollary}

\begin{proof}
    A point $p \in \rL$ is smooth if and only if the codimension of the
    tangent space equals $\,\codim(\rL)= n-d$. The description of the
    tangent space in terms of the matroids $M(A_J)$ and $M(A/J)$ in Corollary~\ref{cor:cotangent} 
    shows  that its codimension is $\,|J| - \rk A_J + {\rm Par}(A/J)\,$ where
    ${\rm Par}(A/J)$ is the  dimension of the space of $2$-circuits of $A/J$.
    Suppose $M(A/J)$ has $r$ distinct $1$-flats (or lines), and let 
     $\lambda_1,\dots,\lambda_r$ be the sizes of these parallelism classes.
    The circuits of each parallelism class span a linear space of
    dimension $\lambda_i - 1$. As these circuits are disjoint, we have
    ${\rm Par}(A/J) = \sum_{i=1}^r (\lambda_i - 1) = |J^c| - r$. The number of parallelism
    classes of $M(A/J)$ is at least $\rk(A/J)$. Thus 
    the codimension of the tangent space is $\le n - d$, and equality holds if
    and only if $M(A/J)$ is basic (cf.~Proposition~\ref{rmk:MuOne}).

    Finally, to see that the singular locus is pure of codimension $2$, we
    note that if $M$ is any non-basic matroid of rank $r \ge 3$, then there is
    an element $e$ such that $M/e$ is non-basic. To show this, we can assume
    that  $M$ is non-basic on $r+1$ elements, each representing a different line.
         If $M = M_1 \oplus M_2$ is not
    connected and $M_1$ is non-basic then any $e \in M_2$ will work.
    Otherwise, $M$ is a uniform matroid and the contraction is uniform
     of rank $r -1 \ge 2$ on $r$ elements. 
     By Example \ref{ex:uniform},     the  uniform
    matroid $U_{n,d}$ is non-basic if and only if $n > d > 1$.  Therefore,
    if $J$ is a flat of $M(A)$ such that $M(A/J)$ is non-basic of rank $\ge
    3$, then there is a flat $J^\prime \supset J$ such that $M(A/J^\prime)$ is
    non-basic. 
\end{proof}

\section{Ramification Locus}\label{sec:ramification}

The entropic discriminant describes the locus of points ${\bf b} \in
\PP^{d-1}$ such that the zero-dimensional scheme defined by the constraints
${\bf x} \in \rL$ and $A {\bf x} = {\bf b}$ is not reduced. Equivalently, the
entropic discriminant is the defining polynomial of the branch locus of the
map $\,A:\L^{-1}\rightarrow \mathbb{P}^{d-1}$.  We begin with the observation
that this map has no base points and is hence a projective morphism.

\begin{lemma} \label{lem:nobase}
    The variety $\rL$ is disjoint from the center of the projection $A :
    \mathbb{P}^{n-1} \dashrightarrow \mathbb{P}^{d-1}$.
\end{lemma}
\begin{proof}
    Our claim states that  $\,  \L^{-1}  \cap  \ker(A) =  \{ 0 \}$ holds in
    $\C^n$.  Let $p$ be a vector in $\,  \L^{-1}  \cap  \ker(A)$ and  $J =
    \supp(p)$.  Then  $p_J^{-1} =  z\,A_J$ for some $z \in \C^d$, and $p \in
    \ker(A)$ implies $\, 0  = z \,(A \,p)= (z\,A)\,p = \sum_{j\in J}
    {p_j}^{-1}p_j = |J|$.  It follows that $J  = \emptyset$ and $p = 0$.
\end{proof}

We now focus on the {\em ramification locus} of the dominant projective
morphism $\,A:\L^{-1}\rightarrow \mathbb{P}^{d-1}$.  By definition, this is
the   Zariski closure of the set of regular points $p\in \rL$ for which
\begin{equation}\label{eq:ram} 
    \L \ + \ \rowspan\,\mathrm{Jac}(\L^{-1})(p) \ \not= \ \C^n. 
\end{equation} 
Here $\mathrm{Jac}(\L^{-1})$ is the Jacobian matrix of $\L^{-1}$, whose row
vectors are $\nabla h_v(\x)$ for $v\in \mathcal{C}(A)$, as in
\eqref{eq:circuitpoly}.  This condition states that the intersection of $\rL$
and $\{\x : A\x = Ap\}$ is not transverse at $p$.

The {\em ramification scheme} $\,\sR_A \,=\, \mathrm{Proj}(\C[\x]/ J_A) \,$ is
defined by the following ideal in $\C[x_1,\ldots,x_n]$:
\begin{equation}\label{eq:RA}
J_A \; = \; \bigl( I(\L^{-1}) \,+\, \bigl\langle n\times n\text{ minors of }\begin{pmatrix}A\\ \mathrm{Jac}(\L^{-1})\end{pmatrix}\bigl\rangle \bigr) :  
\bigl\langle \,  (n{-}d)\times (n{-}d)\text{ minors of }\mathrm{Jac}(\L^{-1}) \,\bigr\rangle^\infty.
\end{equation}
By the Zariski-Nagata Purity Theorem \cite{Nag}, the ramification locus is
pure of codimension~$1$ in $\L^{-1}$.  Hence the ramification scheme $\sR_A$
is either empty or has codimension~$1$ in $\L^{-1}$.  The former happens when
$A$ is basic, and the latter happens when $A$ is non-basic.  We prove in
Section \ref{sec:hopefully} that $\sR_A$ contains the singular locus of
$\L^{-1}$ and hence that the saturation  step in \eqref{eq:RA} is redundant.

\begin{definition} \label{def:ED} 
    Let $A \in \R^{d \times n}$ be a non-basic matrix of rank $d$.  The
    \emph{ramification cycle} is the algebraic cycle of dimension $d-2$ in
    $\PP^{n-1}$ defined by the ramification scheme $\sR_A$.  By
    Corollary~\ref{cor:basic}, the push-forward of the ramification cycle
    under the morphism $A : \mathcal{L}^{-1} \rightarrow \PP^{d-1}$ is a cycle
    of codimension $1$. We define the {\em entropic discriminant}  of $A$ to
    be the homogeneous polynomial $H_A(\b)$ that represents this cycle in
    $\PP^{d-1}$. 
    It is unique up to multiplication by a
    non-zero constant.
\end{definition}

The following example shows that the ramification cycle may not be reduced.

\begin{example} 
Let $A$ be the matrix  in Example~\ref{ex:zweivier}. For $a\neq0,2,3$,
the prime ideal of $\L^{-1}$ equals
$$  I(\L^{-1})\,\, =\,\, \langle  \; 2x_1x_2-3x_1x_3+x_2x_3,\, 2x_1x_2-ax_1x_4+(a-2)x_2x_4, \,
3x_1x_3-ax_1x_4+(a-3)x_3x_4 \;\rangle.
$$
The ramification ideal $J_A$ is the sum of $I(\mathcal{L}^{-1})$ and the ideal of
 $4\times 4$ minors of the matrix 
\[\begin{pmatrix} A \\ \mathrm{Jac}(\L^{-1})\end{pmatrix} \;\;\;\;\;= \;\;\;\;\; 
\begin{small}
\begin{pmatrix}
1& 1& 1& 1\\ 
0& 2& 3& a\\ 
2x_2-3x_3& 2x_1+x_3& -3x_1+x_2& 0\\
 2x_2-ax_4& 2x_1-(a-2)x_4& 0& -ax_1+(a-2)x_2\\ 
3x_3-ax_4& 0& 3x_1+(a-3)x_4& -ax_1+(a-3)x_3\\ 
0& x_3-(a-2)x_4& x_2+(a-3)x_4& -(a-2)x_2+(a-3)x_3
\end{pmatrix}\end{small}.\]
The ramification cycle is
a zero-dimensional cycle of degree $4$ in $\PP^3$.
For the special value $a=6$, it is twice the reduced cycle of degree $2$
defined by $ \langle 2x_2-3x_3+6x_4,2x_1-x_3+4x_4,x_3^2-4x_3x_4+8x_4^2 \rangle $.
The push-forward of this cycle under $\PP^3 \dashrightarrow \PP^1$ 
is defined by the binary quartic in Example~\ref{ex:zweivier}.
 \hfill$\diamond$
 \end{example}

Since the projection $A : \PP^{n-1} \dashrightarrow \PP^{d-1}$ has no base
points on the subscheme $\sR_A$ (by Lemma \ref{lem:nobase}), the
push-forward by $A$ preserves the degree of the ramification cycle. Thus, in
order to establish the degree formula in Theorem \ref{thm:intro}, it suffices
to show that the degree of $\sR_A$ equals  $\,2 (-1)^d(d \chi(0) + \chi'(0))$.
In order to compute its degree, we use a slightly different description of
$\sR_A$.  Let $T$ denote the dense torus $\{x_1 x_2 \cdots x_n \not= 0\}$ in
the projective space $\PP^{n-1}$. Inside $T$,  the variety $\L^{-1}$ is a
complete intersection. Namely, it is defined by $B\cdot\x^{-1}=0$, where $B =
(B_1, \hdots , B_n)$ is a {\em Gale transform} for $A$, that is, an $(n-d)\times
n$-matrix whose rows span the kernel of $A$. Consider the polynomial
\[ g_A(\x) \;\;\; = \;\;\; 
\det \begin{pmatrix}
\multicolumn{3}{c}{A}\\
B_1x_1^{-2} &\!\! \cdots \!\! &B_nx_n^{-2} 
\end{pmatrix}\cdot \prod_{i=1}^nx_i^2 
\;\;\;=\;\;\;\det \begin{pmatrix}
A_1x_1^{2} &\cdots &A_nx_n^{2} \\
\multicolumn{3}{c}{B}
\end{pmatrix}.
 \]
 The $n \times n$-matrix above now plays the same role as the Jacobian matrix did in (\ref{eq:RA}).
Thus the hypersurface defined by $g_A(\x) = 0$ inside $\L^{-1} \cap T$ is the restricted
ramification locus $\,\sR_A \cap T$.

If $g_A$ is zero at a point $p \in T$ then the 
intersection $\ker(A \diag(p)^2) \cap \ker(B)$ contains a non-zero vector.
The kernel of $B$ is spanned by the rows of $A$, 
so the $d\times d$-matrix  $A\diag(p)^2A^T$ also drops rank.
Hence $g_A(\x)$ divides ${\rm det}(A\diag(\x)^2A^T)$. Both polynomials
have the same degree $2d$,  and hence they are equal (up to a scalar,
which we ignore). Using the 
Cauchy-Binet Formula, this gives
\begin{equation}\label{eq:gA2}
g_A(\x) \;\; = \;\; \det( A \diag(\x)^2 A^T) \;\; = \;\; 
\sum_{I\in\binom{[n]}{d}}\det(A_I)^2\prod_{i\in I}x_i^2.
\end{equation}
We next define
 similar  polynomials that cut out $\sR_A$ on the non-singular boundary strata of $\L^{-1}$.
 Let $J \subset [n]$ be any proper flat of rank $r$ in $M(A)$ and set
${\bf x}_J = (x_j\;:\; j\in J)$. Let 
 $\hat{A}_J$ now denote any $r\times |J|$ submatrix of $A_J=(A_j\;:\;j\in J)$
whose rows are linearly independent. We define
\begin{equation}\label{eq:gAJ}
g_{A_J}(\x_J) \;\; = \;\; \det( \hat{A}_J \diag(\x_J)^2 \hat{A}_J^T) \;\; = \;\; 
\sum_{I\in\binom{J}{r}}\det( \hat{A}_I)^2\prod_{i\in I}x_i^2.
\end{equation}
Here $\hat{A}_I$ denotes the square submatrix of $\hat{A}_J$
induced on the $r$ columns indexed by $I \subset J$.

\begin{lemma}\label{thm:RAnonsing}
 Let $p$ be a smooth point on the reciprocal plane $\rL$ with $\supp(p)=J$.  Then
 the ramification locus $\sR_A$ contains the point $p$ if and only if $g_{A_J}(p_J) = 0$. 
 \end{lemma}

\begin{proof}
Since $p$ is smooth,  the condition \eqref{eq:ram} reduces to
$\ker(A_J) \cap \ker({\rm Jac}(\rL_J)) \neq \{0\}$. From the argument 
prior to (\ref{eq:gA2}) we see that,
 for $p_J\in (\C^*)^J$, this is equivalent to $g_{A_J}(p_J)=0$.
\end{proof}

\begin{remark}
This characterization shows that the ramification locus $\sR_A$ 
equals the closure of its intersection with the torus, $\sR_A\cap T$.  
To see this, suppose that $\sR_A$ has some component $Z$ contained in the boundary of the torus 
$\{x_1\cdots x_n=0\}$.  Then $Z$ is contained in $\rL_J$ for some proper flat $J$, where 
$\dim(\rL_J) = {\rm rank}(A_J)-1$. Since $\sR_A$ is pure of codimension one in $\L^{-1}$, 
we see that $\dim(Z)=d-2$. It follows that ${\rm rank}(A_J) = d-1$ and $Z=\rL_J$. 
However, $M(A/J)$ has rank $1$ and is therefore basic.  Lemma~\ref{thm:RAnonsing} then
tells us that $\rL_J$ is not contained in $\sR_A$. 
This shows that to define the ideal $J_A$ in \eqref{eq:RA}, 
we could instead saturate with respect to the ideal $\langle x_1 x_2 \cdots x_n \rangle$.
\end{remark}

We shall now use the polynomial $g_A(\x)$ to compute the degree of the ramification cycle.

\begin{proof}[Proof of Theorem \ref{thm:intro}]
Let $A$ be a non-basic real $d \times n$-matrix  of rank $d$ 
and $\chi(t)$ the characteristic polynomial of the matroid $M(A)$.
We shall prove that the degree of the algebraic cycle underlying the $(d-2)$-dimensional
subscheme $\sR_A$ of $\PP^{n-1}$ equals the matroid invariant (\ref{eq:deg}).
Lemma \ref{lem:nobase} then implies that $H_A(\b)$ has the same degree, and this will
complete the proof of Theorem \ref{thm:intro}.

From above, we know that the scheme $\sR_A$ is contained in the hypersurface 
$\{g_A = 0\}$ of $\PP^{n-1}$. Let $\widehat{\sR}_A$ denote the
scheme-theoretic intersection of the reciprocal plane with this hypersurface:
\begin{equation}\label{eq:bigintersection}
 \widehat{\sR}_A \;\; = \;\; \mathrm{Proj} \bigl(\, \C[\x] \,/\, (I(\L^{-1})+\langle g_A \rangle )\,\bigr).
 \end{equation}
 The $(d-2)$-dimensional scheme  $\widehat{\sR}_A$ is the intersection of
   the $(d-1)$-dimensional irreducible variety $\mathcal{L}^{-1}$
   and the hypersurface $g_A$.
   By B\'ezout's Theorem  \cite[Thm. 1.4.4]{vogel}, its degree equals
 \begin{equation}
 \label{eq:bezout}
  \deg \bigl( \widehat\sR_A \bigr)  \quad = \quad 
    \deg(g_A)\cdot \deg(\L^{-1}). 
    \end{equation}
    
We claim that $ \widehat{\sR}_A$ decomposes into $\#{\rm Hyp}(A)+1$ components of dimension
$d-2$, one of which is $\sR_A$. Here ${\rm Hyp}(A)$ denotes the set
of  hyperplane flats, that is, flats $J$ such that $\rk(A_J)=d-1$. 
We see that $\sR_A$ and $\widehat{\sR}_A$ agree in the torus $T$, so their difference must
lie in the coordinate hyperplanes. 
Recall from Section~\ref{sec:reciprocalplanes} that the reciprocal plane 
intersects the dense torus $T^J $ of $ \PP^J$ if and only if $J$ is a flat, and if so,  the closure of that intersection
is the reciprocal plane $\rL_J$. Such a stratum has dimension $d-2$ in $\PP^{n-1}$ if and only if $J$ is 
a hyperplane flat. Since $J\in \rm{Hyp}(A)$ does not contain 
a basis of $M(A)$, each summand in the formula \eqref{eq:gA2} for
$g_A$  vanishes on $T^J$. To be precise, $g_A$ vanishes 
to order exactly $2$ on the torus $T^J$, since
$J$ is only one element away from containing a basis.

Furthermore,   the strata $\rL_J$ are not contained in $\sR_A$ for $J\in \rm{Hyp}(A)$.
This follows from Lemma~\ref{thm:RAnonsing}. Indeed,
 by Corollary~\ref{thm:singLA}, the points in $\L^{-1}\cap T^J$ are non-singular in $\L^{-1}$,
 and hence the polynomial $g_{A_J}(\x_J)$ is not identically zero on $\rL_J$.
  We conclude that the irreducible varieties
   $\rL_J$, for $J\in \rm{Hyp}(A)$, are  components
   of dimension $d-2$ and multiplicity $2$ in the scheme $\widehat{\sR}_A$.
   
   We have derived the following equidimensional decomposition of the cycle
   defined in (\ref{eq:bigintersection}):
   \begin{equation}\label{eq:bigint2}
 \widehat{\sR}_A \;\;= \;\;\sR_A \;\cup\; \left( \bigcup_{J\in{\rm Hyp}(A)}
 2\cdot \rL_J \right).
\end{equation}
Since the degree is additive on equidimensional cycles, we can use (\ref{eq:bezout}) to conclude that
\begin{equation}
\label{eq:degsRA}
    \deg(\sR_A)  \quad = \quad 
    \deg(g_A)\cdot \deg(\L^{-1}) \;-\; 2\!\!\!\sum_{J \in {\rm Hyp}(A)}\!\!\!
    \deg(\rL_J) \quad = \quad
    2d \cdot \mu(A) \;-\; 2\!\!\!\sum_{J\in {\rm Hyp}(A)}\!\!\!\mu(A_J).
\end{equation}
The coefficient of $t^i$ in the characteristic polynomial 
$\chi(t)$ equals $(-1)^{d-i}$
times the sum of the M\"obius invariants  $\mu(A_J)$ 
where $J$ runs over all flats of rank $d-i$.
For $i = 0$ this gives
$ \mu(A) = (-1)^d  \chi(0) $, and for $i=1$
we get $\,\sum_{J\in {\rm Hyp}(A)}\! \mu(A_J) =  (-1)^{d-1} \chi^\prime(0)$.
 Hence the right hand side of (\ref{eq:degsRA}) equals the
 desired matroid invariant \eqref{eq:deg}.   This completes the proof of
   Theorem \ref{thm:intro}.
     \end{proof}

The decomposition (\ref{eq:bigint2}) can be used to compute
the ideal of the ramification scheme. Namely, since all
hyperplane strata $\rL_J$ lie in 
    complement    of the torus $T$, we have the algebraic identity
 \begin{equation}
 \label{eq:computeRA}
  J_A \,\, = \,\, \bigl( I(\L^{-1}) + \langle \,g_A \,\rangle \bigr) :
    \langle x_1 x_2 \cdots x_n \rangle^\infty. 
  \end{equation}
We illustrate the identity
(\ref{eq:computeRA}) and our proof of Theorem \ref{thm:intro} for the codimension $1$ case.

\begin{example} \label{ex:directA} 
Let $A$ be the matrix in equation (\ref{eq:specialmatrix})
of Section \ref{sec:n=d+1}.
The reciprocal plane $\L^{-1}$ is the hypersurface
defined by the elementary symmetric polynomial $\,e_{n-1}(x_1,x_2,\ldots,x_n) $.
The equation (\ref{eq:gA2}) defining the ramification locus
in the torus is $ \,g_A = e_{n-1}(x_1^2,x_2^2,\ldots,x_n^2) $.
The scheme  $ \widehat{\sR}_A$ in (\ref{eq:bigintersection})
is the complete intersection of these two hypersurfaces. Its ideal has the primary decomposition
\begin{equation} \begin{matrix}
& \big\langle e_{n-1}(x_1,x_2,\ldots,x_n) ,  e_{n-1}(x_1^2,x_2^2,\ldots,x_n^2)
\big\rangle &  &  \\
= & \langle e_{n-1}(x_1,x_2,\ldots,x_n) , e_{n-2}(x_1,x_2,\ldots,x_n) \rangle
& \cap &  \bigcap_{1 \leq i < j \leq n} \langle\, x_i^2 , x_i+ x_j \rangle .
\end{matrix} \end{equation}
This is the decomposition  discussed after (\ref{eq:bezout}),
with the first intersectand being the ideal $J_A$ that defines $\sR_A$.
This ideal is contained in the Jacobian ideal of  the reciprocal plane $\L^{-1}$ because
$$ e_{n-2} \,\,\, = \,\,\, \frac{1}{2} \sum_{i=1}^n \frac{ \partial e_{n-1} }{ \partial x_i} .$$
This identity proves the ideal-theoretical inclusion $J_A \subset I({\rm Sing}(\L^{-1}))$.
We conclude that $\sR_A$ contains ${\rm Sing}(\L^{-1}) $ 
when $n= d+1$. As we shall see in
Theorem \ref{tm:main},
the inclusion  ${\rm Sing}(\L^{-1}) \subset \sR_A $ is always true, 
even if $n > d+1$. 
This inclusion implies, as argued in Corollary \ref{cor:reallylast}, 
that the real variety of $H_A(\b)$ is indeed the union of codimension $2$ planes
given  in (\ref{eq:d+1choose2}). \hfill $\diamond$
\end{example}

We close this section with
a combinatorial proof of the assertion, stated informally immediately after Theorem
\ref{thm:intro}, that generic matrices maximize the degree of the entropic
discriminant.

\begin{proposition} \label{prop:upperbound}
As $A$ ranges over all non-basic $d \times n$-matrices
of rank $d$, the degree of the entropic discriminant
$H_A(\b)$ attains its maximal value $2(n-d)\binom{n-1}{d-2}$ when
the matroid $M(A)$ is uniform.
\end{proposition}

\begin{proof}
    It follows from Theorem~\ref{thm:intro} and Example~\ref{ex:uniform} that
    $2(n-d)\binom{n-1}{d-2}$ is  the degree of the entropic discriminant when
    $M=M(A)$ is uniform.  We must show that this number is a strict upper
    bound otherwise.    The claim is an entirely matroid-theoretic statement,
    and so let us define $\delta(M) = 2
    (-1)^{\rk(M)}(\rk(M) \chi_M(0) + \chi_M'(0))$ for all matroids $M$. The
    characteristic polynomial satisfies a deletion-contraction recurrence,
    namely, $\chi_M(t) = \chi_{M \backslash e}(t) - \chi_{M / e}(t)$ for $e
    \in M$ not an isthmus. It follows that the entropic degree satisfies a
    deletion-contraction recurrence plus a correction term:
    \[
        \delta(M) \ = \ \delta(M \backslash e) \ + \ \delta(M / e) \ + \
        \mu(M / e).
    \]
    All three terms on the right hand side are non-negative. The desired inequality follows
    by induction on the rank $d$ and corank $n-d$.      In rank $1$ all
    simple matroids are uniform.     Corank $1$ is dealt with in
    Section~\ref{sec:n=d+1}. The same argument shows that $\mu(M)
    \le \tbinom{n-1}{d-1}$ with equality if and only if $M = U_{d,n}$.
\end{proof}

\section{Real Issues}\label{sec:hopefully}

Our point of departure for this paper was the observation that,
for real $\b$, the equations (\ref{eq:opti2}) have only real solutions, namely,
the $\mu(A)$ analytic centers of the bounded regions
in the arrangement of $n$ coordinate hyperplanes
in $\{A \x = \b\} \simeq \R^{n-d}$. It is thus natural to ask 
what it means for two such analytic centers to collide,
and how this relates to the real points in 
the ramification locus and in the entropic discriminant.
We shall prove that the real loci of these two complex varieties are both pure of codimension one.
Our first step in this direction  is the following lemma.

\begin{lemma} \label{prop:ramsing}
All real points in the ramification scheme 
are singular in the reciprocal plane:
\begin{equation}
\label{eq:sRsing}
       (\sR_A)_\R \,\,\, \subseteq \,\,\, \sing(\rL)_\R .
\end{equation}
\end{lemma}

\begin{proof}
The sum of squares formula in (\ref{eq:gA2}) reveals 
that $g_A(\x) = 0$ has no real solutions in the torus $T$. In symbols,
$(\sR_A \cap T)_\R = \emptyset$. Likewise, for any flat $J$ with $\rL_J$ nonsingular in $\L^{-1}$,
the polynomial $g_{A_J}$ is a similar sum of squares, and hence
$(\sR_A \cap T^J)_\R = \emptyset$. Lemma~\ref{thm:RAnonsing} ensures that no
regular point of $\mathcal{L}^{-1}$ with real coordinates lies in the
ramification locus of the morphism $ A : \mathcal{L}^{-1} \rightarrow \PP^{d-1}$.
\end{proof}

The following is our main result in this section. We find 
 that the reverse inclusion  holds in~(\ref{eq:sRsing}).

\begin{theorem}\label{tm:main}
    The ramification scheme $\sR_A$ contains the singular locus of $\rL$, and we have
    \begin{equation}
\label{eq:sRsing2}
      (\sR_A)_\R \,\, = \,\, \sing(\rL)_\R . 
      \end{equation}
\end{theorem}

This theorem implies that the saturation in the formula \eqref{eq:RA} 
 for the ramification ideal $J_A$ was unnecessary. Before presenting
 the proof, we shall derive two corollaries and discuss one example.

\begin{corollary} 
\label{cor:almostlast}
The Zariski closure of  $(\sR_A)_\R$ is pure
of codimension $2$ in $\mathcal{L}^{-1}$. 
\end{corollary}

\begin{proof}
Theorem \ref{tm:main} implies that the Zariski closure of
the set $(\sR_A)_\R$
of real ramification points equals the singular locus 
${\rm Sing}(\mathcal{L}^{-1})$  of the reciprocal plane $\mathcal{L}^{-1}$.
Corollary~\ref{thm:singLA} represents
${\rm Sing}(\mathcal{L}^{-1})$ as a union of linear spaces
all of which are defined over $\R$ and have codimension $2$ in $\mathcal{L}^{-1}$.
\end{proof}

We now obtain the following characterization of
the real locus of the entropic discriminant.

\begin{corollary}
\label{cor:reallylast}
The Zariski closure of the set of real points in the hypersurface
defined by the entropic discriminant $H_A(\b)$
is pure of codimension $2$ in $\PP^{d-1}$.
Its irreducible components are the  linear spaces
$\,{\rm span}(A_j : j \in J)$, where
$J$ runs over all non-basic corank $2$ flats of $M(A)$.
\end{corollary}

\begin{proof}
The real variety of $H_A$ is the image of the real points in $\sR_A$
under the $\mu(A)$-to-one morphism 
$A : \mathcal{L}^{-1} \rightarrow \PP^{d-1}$. Hence the real variety
of $H_A$ is pure of codimension $2$ in $\PP^{d-1}$ as well. The description
of its irreducible components now follows from that given in
Corollary~\ref{thm:singLA}.
\end{proof}

We now revisit our very first example to illustrate the previous corollary.

\begin{example} 
    For $d=3$, the codimension-2 strata of $\L^{-1}$ are the $n$ coordinate 
points $e_i$ in $\PP^{n-1}$.  Their images under the map $A$ are the columns 
$A_1, \hdots, A_n$. For generic $A$, the points $e_1,\ldots,e_n$ comprise $\sing(\L^{-1})$.
Lemma \ref{prop:ramsing} implies that
$V_\R(H_A) $ is contained in $\{A_1,\ldots,A_n\}$,
and Theorem~\ref{tm:main} reveals that equality holds.
For special $3 \times n$-matrices $A$, the matroid $M(A/i)$ may be basic for some $i$. 
If this happens then $e_i$ is
a non-singular point in $\L^{-1}$ and
its image $A_i$ does not 
belong to $V_{\R}(H_A)$. 
Looking back at Example~\ref{ex:dreifunf}, 
we notice that the matroid $M(A/i)$ is basic for $i=1$ and 
it is non-basic for $i=2,3,4,5$.
This explains our finding in (\ref{eq:finding}) that the real variety $V_{\R}(H_A)$ consists of 
precisely the four points 
$A_2,A_3, A_4 $ and $ A_5$ in the projective plane $\PP^2$. \hfill $\diamond$
\end{example}

We are now ready to present the proof of our main result in this section.

\begin{proof}[Proof of Theorem \ref{tm:main}]
We first note that the identity (\ref{eq:sRsing2}) follows immediately from
Lemma \ref{prop:ramsing} and the inclusion $\, \sR_A \, \supseteq \, \sing(\rL)\,$
 in the first assertion. Hence it suffices to prove that inclusion.

By Corollary~\ref{thm:singLA}, the singular locus of $\rL$ is a
reducible variety whose irreducible components are the boundary strata 
$\rL_J$ where $M(A/J)$ is a non-basic matroid of rank 2.  We consider one such
 component $\rL_J$, regarded as a subvariety
 of $\C^J \times \{0\}$ inside of  $\C^n = \C^J \times \C^{J^c}$.
  A generic point of $\rL_J$ has the form $(p,0) $ where $p\in  (\C^*)^J$. 
  Our goal is to show that this point lies in the ramification locus $\sR_{A}$
  by producing a sequence of points in $\sR_A$ that converges to $(p,0)$.

We may assume that $J = \{1,\ldots,k\}$ is a
flat of rank $d-2$ and our matrix $A$ has the block form
\[
    A \ = \ \left(%
    \begin{array}{c|c}
        \hat{A} & * \\
        \hline
         0 & B \\
      \end{array}
      \right)
\]
where $\hat{A} \in \R^{(d-2) \times k}$ and $B \in \R^{2 \times (n-k)}$ are
both of full row-rank. In these coordinates, we get $M(A_J) = M(\hat{A})$ and
$M(A/J) = M(B)$. 

Now, let us return to our generic point $(p,0) \in \rL_J$. 
 The partial specialization $g_A(p, \x_{J^c})$
  is a polynomial in $ \C[x_{k+1},x_{k+2},\dots,x_n]$. It is non-homogeneous 
and its terms of lowest total degree come from those bases $I$ of
$M(A)$ for which $|I \cap J| =d-2$. For any such $I$, we have
\[
\det(A_I)  \ = \ \det(\hat{A}_{I\cap J}) \cdot \det(B_{I \cap J^c}).
\]
From this we see that the initial form of $g_A(p,\x_{J^c})$ of lowest degree terms can be written as
\begin{equation}\label{eq:ingA}
    \mathrm{in}_{-\1}( g_A(p,\x_{J^c})) \;\; = \;\;
    g_{\hat{A}}(p)\cdot g_{B}(\x_{J^c}). 
\end{equation}
From the results of Section~\ref{sec:ramification} we know that $\{g_{\hat{A}} = 0 \}
\cap \rL_J $ has codimension $1$ in $\rL_J$.
This implies  $g_{\hat{A}}(p) \not=0$ because the point $(p,0)$ was chosen to be generic in $\rL_J$.

In order to proceed, we need to represent the ramification locus around $p$ 
by a single polynomial, rather than as a subvariety of $\L^{-1}$. 
To do this, we rationally parametrize the points $\x_{J^{c}}$ 
for which $(p, {\bf x}_{J^{c}})$ lies in $\rL$ using the matrix $B$.  
First, note that the intersection of the linear space $\L$ with 
$\{p\}\times \C^{J^c}$ gives an affine linear space in $\C^{J^c}$
of the form $v + \rowspan(B)$ for some vector $v$ in $\C^{J^c}$.  We can parametrize
this space by $v+{\bf z}B$ where $\z = (z_1,z_2)$. This gives the rational parametrization 
$(p, (v+{\z}B)^{-1})$ of the intersection of $\rL$ with $\{p\}\times \C^{J^c}$.

Now we plug this parametrization into $g_A(p,{\bf x}_{J^c})$ and clear 
denominators to get a polynomial in $\C[z_1, z_2]$. Define $g(\z) \in \C[z_1, z_2]$ to be this polynomial,
\begin{equation}\label{eq:moregs}
    g(\z) \ \ = \   {g_A}(p,(v + \z B)^{-1})  \prod_{i\in J^c} (v_i + \z B_i)^2   \; 
    = \; \sum_{I\in \binom{[n]}{d}} \det(A_I)^2 
 \prod_{i\in I\cap J}p_i^2  \prod_{j \in J^c \backslash I} (v_j + \z B_j)^2.
\end{equation}
If $\z$ is a solution to $g({\bf z})=0$ for which each coordinate of $v+{\bf z}B$ is non-zero, 
then the point $(p, (v+{\bf z}B)^{-1})$ lies in the ramification locus $\sR_A$.

Since $J$ is a flat, the $n-k$ linear forms $\z B_i$ are non-zero 
for all indices $i$. This implies that $x_i = 1/(v_i+\z B_i)$ has degree $-1$. Thus the terms of highest degree
in $g(\z)$ correspond exactly to the terms of lowest degree in $g_A(p,\x_{J^c})$. From \eqref{eq:ingA}, we see 
that the leading form of $g(\z)$ is 
\[ 
{\rm in}_{\1}(g(\z)) \;\;\;=\;\;\;  
 g_{\hat{A}}(p)\cdot g_{B}( (\z B)^{-1})\cdot \prod_{i\in J^c} (\z B_i)^2.
\]
Our next step is to find a solution to 
the initial equation ${\rm in}_{\1}g(\z)=0$ and
to then extend it to the desired sequence of points in $\sR_A$.
As the matroid $M(A/J) = M(B)$ is non-basic, it follows from Corollary~\ref{cor:basic} that the
ramification $\sR_{A/J}$ is nonempty.  Hence there is a point $q \in \L_{A/J}\cap (\C^*)^{n-k}$ such that
$g_{A/J}(q^{-1})=g_{B}(q^{-1}) =0$. Let $z$ be the unique vector such that $z B = q$. 
We may assume that $B$ has the form $(\,\mathrm{Id}_{2}\,\, B^\prime\,)$. Thus implying that  
$z_i = q_i \not=0$ for $i = 1,2$.

By Lemma~\ref{lem:pui} below, we can extend this solution $z\in (\C^*)^2$ to a solution $Z = Z(\epsilon)$ of $g(\z)$, 
where the coordinates of $Z = (Z_1, Z_2)$ lie in the field $\C\{\!\{\epsilon\}\!\}$ of Puiseux series:
\[    Z_i\;\; =\;\;z_i\frac{1}{\epsilon} \;+\; \text{ higher order terms } \;\; \in \;\; \C\{\!\{\epsilon\}\!\}  
\;\;\;\;\;\; \text{ for }i=1,2.\]
Moreover, by Lemmas~\ref{lem:pui2} and \ref{lem:pui}, these series converge in a neighborhood of zero in $\R_{>0}$. 

Now consider the point $Q = Q(\epsilon) = v+ZB$ with coordinates
$Q_i  = q_i \frac{1}{\epsilon}  +  \cdots$ in  $\C\{\!\{\epsilon\}\!\}$.
We can invert $Q_i$ in the field of Puiseux series to get 
\[ Q_i^{-1} \;\;\; = \;\;\; q_i^{-1} \epsilon \;+\; \text{ higher order terms } \;\; \in \;\; \C\{\!\{\epsilon\}\!\}
\;\;\;\;\;\; \text{ for }i=1,\hdots,n-k,\]
and these series
 converge for real $\epsilon$ in an open segment $(0,\epsilon_0)$ near zero (see Lemma~\ref{lem:pui2} below). 

Then, by $\eqref{eq:moregs}$, the point $(p, Q^{-1})$ in $\rL \otimes_{\C} \C\{\!\{\epsilon\}\!\}$ is a zero of the
polynomial $g_A({\bf x})$.  Specializing to sufficiently small $\epsilon\in \R_{>0}$, gives a point 
$(p,Q(\epsilon)^{-1})\in (\C^*)^n$ that belongs to the
ramification locus $\sR_A$. Furthermore, as $\epsilon$ approaches $0$, the limit of the points
$(p, Q^{-1}(\epsilon))$ is $(p,0)$ in $\rL_J \times \{0\}$.
This shows $\rL_J \subseteq \sR_A$ and consequently
$\sing(\rL) \subseteq \sR_A$.
\end{proof}

Before Lemma~\ref{lem:pui}, we need a short lemma on the convergence of reciprocals of Puiseux series. 

\begin{lemma}\label{lem:pui2}
If $x(\epsilon)$ is a nonzero Puiseux series that converges for $\epsilon>0$ in a neighborhood of $0$, then its inverse $x(\epsilon)^{-1}$ in $\C\{\!\{\epsilon\}\!\}$ also converges for real $\epsilon$ in an
open segment $(0,\epsilon_0)$.
\end{lemma}

\begin{proof}
Suppose $x(\epsilon) = u\epsilon^k+\text{ higher order terms}$. 
We can write the field of Puiseux series as the union of $\C((\epsilon^{1/m}))$ over $m\in \Z_+$.
 Thus for some $m \in \Z_+$, replacing $\epsilon$ with $\epsilon^m$ yields a Laurent series 
$x(\epsilon^m)$, which also converges in a neighborhood of 0. In particular, 
$\epsilon^{-mk}x(\epsilon^m)$ is a convergent power series with constant term 
$u$ and has an inverse $y(\epsilon)$ in the 
ring of convergent power series (see \cite[\S 6.4]{Fischer}).  
Then $y(\epsilon) = 1/u+ \cdots$ satisfies $\epsilon^{-mk}x(\epsilon^m)y(\epsilon) = 1$. 
Replacing $\epsilon$ with $\epsilon^{1/m}$, we see that 
$\epsilon^{-k}y(\epsilon^{1/m})$ is an inverse for $x(\epsilon)$. Furthermore, since $y(\epsilon)$ 
and $y(\epsilon^{1/m})$ converge in a neighborhood of zero,
 $x(\epsilon)^{-1} = \epsilon^{-k}y(\epsilon^{1/m})$ also
converges for $\epsilon > 0$ in a neighborhood of zero.  
\end{proof}

Now all that remains is to lift roots of initial forms to solutions over $\C\{\!\{\epsilon\}\!\}$.

\begin{lemma}\label{lem:pui}
Let $g(z_1,z_2)$ be a polynomial with complex coefficients 
and initial form $ {\rm in}_{\bf 1}(g)$, consisting of the highest terms with respect to
 total degree. Let $u = (u_1,u_2) \in (\mathbb{C}^*)^2$ be
 any solution to the equation $ {\rm in}_{\bf 1}(g)(u_1,u_2) = 0$. Then there exists a vector
$v(\epsilon) $ that satisfies $\,g(v(\epsilon)) \,=\, 0\,$ and
 whose coordinates are
Puiseux series of the form
$$ v_i(\epsilon) \,\, = \,\, u_i \frac{1}{\epsilon} \,+\, \hbox{higher order terms in $\epsilon$,}
\quad \qquad \hbox{for} \quad  i=1,2,$$
that converge for $\epsilon$ in some neighborhood $(0,\epsilon_0)$ of zero.
\end{lemma}

\begin{proof}
We invert the variables $z_i$ and work with the polynomial
\[\overline{g}(z) \;\;=\;\; z_1^{{\rm deg}(g)}\cdot z_2^{{\rm deg}(g)} \cdot g(z_1^{-1}, z_2^{-1} ). \]
The highest-degree terms of $g$ then correspond to the lowest-degree terms of $\ol{g}$. 
Furthermore, the point $u^{-1} = (1/u_1, 1/u_2)$ is a solution of $ {\rm in}_{\bf -1}(\ol{g})$.

Our hypothesis states that the Newton polygon of $\ol{g}(z_1,z_2)$ has an edge
of slope $-1$, and $(1/u_1,1/u_2)$ is a root of the corresponding
binary form ${\rm in}_{\bf -1}(\ol{g})(z_1,z_2)$. Using the classical Newton-Puiseux algorithm,
we can construct a power series expansion of $z_2$ in terms
of $z_1 = \frac{1}{u_1} \epsilon$, having the form $z_2 = \frac{1}{u_2}\epsilon + \cdots $.
The resulting series in $\epsilon$ converges by the arguments in
\cite[\S 7.11]{Fischer}.  

This solution has an inverse in the field of Puiseux series, and this inverse will be 
our desired solution $(v_1(\epsilon), v_2(\epsilon))$ of $g(z_1,z_2) = 0$. 
Namely, if $w(\epsilon) = (w_1(\epsilon), w_2(\epsilon)) \in \C\{\!\{ \epsilon \}\!\}^2$ is the solution
to $\ol{g}(z_1, z_2)$ found in the paragraph above, then $v_i(\epsilon) = w_i(\epsilon)^{-1}$
is a solution to $g(z_1, z_2)$. By Lemma~\ref{lem:pui2}, the Puiseux series $v_i(\epsilon)$ converge 
in a neighborhood $(0,\epsilon_0)$ of the origin in $\R_{> 0}$.
\end{proof}

\smallskip

This concludes our study of the entropic discriminant.
In spite of the progress that has been achieved,
there are still many unresolved problems concerning $H_A(\b)$.
We list five open questions:

\medskip

\noindent {\bf Open Questions:}
\begin{enumerate}
\item {\em Is the entropic discriminant $H_A({\bf b})$ always a sum of squares?} \\
We know that the answer is yes for $n=d+1$ and for $d=2$,
but even the case $d=3$ of plane curves is open. It would be especially nice to write $H_{A}(\b)$ as sum of squares in the maximal minors of the matrix $(A,\b)$, as we did in \eqref{eq:plucker} and \eqref{eq:pluck2}
for  $(d,n)=(2,3),\,(2,4)$.\smallskip
\item {\em What is the Newton polytope of the entropic discriminant $H_A({\bf b})$?} \\
For instance, when $A$ is the matrix in (\ref{eq:specialmatrix})
then the table (\ref{eq:somedata})  suggests that
the Newton polytope of $H_A(\b)$ is 
the standard permutohedron, scaled by a factor of two. \smallskip
\item {\em Find ideal generators for the ramification scheme.} \\
Here is a 
 concrete conjecture about minimal generators of the ideal $J_A$
 in (\ref{eq:computeRA}).
 Fix  $n \geq d+2$ and a $d \times n$-matrix $A$ whose matroid is uniform.
We know from Example \ref{ex:uniformreciprocal} that $I(\mathcal{L}^{-1})$ is minimally
generated by $\binom{n-1}{d}$ polynomials of degree $d$. We
conjecture that $J_A$ has precisely $\binom{d+1}{2}$ additional minimal generators
of degree $2d-2$, namely, the restrictions to $\mathcal{L}^{-1}$ of the
    rational functions  $g_A(\x)/x_i x_j$ for some $i,j \in [n]$.
We can show that these rational functions are polynomials
on  $\mathcal{L}^{-1}$ and that they vanish on ${\rm Sing}(\mathcal{L}^{-1}) $.
Do they generate our ideal?
\item {\em How is the entropic discriminant related to  the Gauss curve of the central curve?} \\
The degree formula for the Gauss curve in \cite[\S 5]{DSV}
is essentially the same as the degree formula we derived for $H_A(\b)$.
What is the most natural geometric explanation for this?
\item {\em How does the entropic discriminant depend on the                     
choice of monomial to be maximized?} \\
In light of Varchenko's work \cite{varchenko}, it is natural to replace
$x_1 x_2 \cdots x_n$ in (\ref{eq:opti1}) by a monomial
$\mathbf{x}^\mathbf{u} = x_1^{u_1} x_2^{u_2} \cdots x_n^{u_n}$
with indeterminate exponents. This would lead to a refined
discriminant that is a bihomogeneous polynomial in
$(\mathbf{b},\mathbf{u})$. What is the bidegree of that polynomial?
\end{enumerate}

\bigskip \bigskip

\noindent
{\bf Acknowledgments.} 
We thank Anders Bj\"orner, Igor Dolgachev, Chris Hillar, Daniel Plaumann, Frank Sottile and David Speyer 
for helpful discussions. 
We are especially grateful to David Speyer for helping us to
a proof of Theorem \ref{tm:main}.
Raman Sanyal was supported by a Miller Postdoctoral Research Fellowship
at UC Berkeley. Bernd Sturmfels and Cynthia Vinzant were partially supported by the
U.S.~National Science Foundation (DMS-0757207 and DMS-0968882). Bernd Sturmfels also 
thanks the Mittag-Leffler-Institute
and MATHEON Berlin for their hospitality during this project.

\bigskip


\begin{thebibliography}{10}

\bibitem{AS} 
{\sc P.~Alexandersson and B.~Shapiro}:
{\em Discriminants, symmetrized graph monomials and sums of squares},
Experimental Math.~{\bf 21} (2012) 353--361.

\bibitem{bor46}
{\sc C.~W. Borchardt}: {\em Neue Eigenschaft der Gleichung, mit deren H\"ulfe
  man die secul\"aren St\"orungen der Planeten bestimmt.}, J. Reine Angew.
  Math. {\bf 30} (1846) 38--45.
  
\bibitem{DSV}
{\sc J.~A. de~Loera, B.~Sturmfels, and C.~Vinzant}:
{\em The central curve in linear programming},
Found.~Comput.~Math.~{\bf 12} (2012) 509-540.

\bibitem{DP}
{\sc A.~Dimca and S.~Papadima}: {\em Hypersurface complements, {M}ilnor fibers
  and higher homotopy groups of arrangments}, Annals of Mathematics {\bf 158} (2003) 473--507.

\bibitem{dolbook}
{\sc I.~V. Dolgachev}: {\em Classical Algebraic Geometry: A Modern View},
Cambridge Univ. Press, 2012.

\bibitem{Dol}
{\sc I.~V. Dolgachev}: {\em Polar {C}remona
  transformations}, Michigan Mathematical Journal {\bf 48} (2000) 191--202.
  
\bibitem{dom10}
{\sc M.~Domokos}: {\em The discriminant of symmetric matrices as a sum of squares and the orthogonal group}, Communications on Pure and Applied Mathematics
{\bf 64} (2011) 443--465.
  
\bibitem{Fischer} {\sc G.~Fischer}: {\em Plane Algebraic Curves},
{vol.~15 of Student Mathematical Library}, American Mathematical Society, {Providence, RI}, {2001}.

 \bibitem{vogel}
{\sc H.~Flenner, L.~O'Carroll, and W.~Vogel}:
{\em Joins and Intersections}, Springer Verlag, New York, 1999.

\bibitem{GreZas}
{\sc C.~Greene and T.~Zaslavsky}, {\em On the interpretation of {W}hitney
numbers through arrangements of hyperplanes, zonotopes, non-{R}adon
partitions, and orientations of graphs}, Trans. Amer. Math. Soc., 280
(1983) 97--126.

\bibitem{HilWib}
{\sc C.~Hillar and A.~Wibisono}: {\em  
Maximum entropy distributions on graphs},
{\tt arXiv:1301.3321}.

\bibitem{Huh}
{\sc J.~Huh}: {\em Milnor numbers of projective hypersurfaces and the chromatic
  polynomial of graphs},  J.~Amer.~Math.~Soc. {\bf 25} (2012) 907--927.

\bibitem{huh2}
{\sc J.~Huh and E.~Katz}: {\em Log-concavity of characteristic polynomials and
  the Bergman fan of matroids},  Mathematische Annalen {\bf 354} (2012) 1103-1116.
  
\bibitem{ilyu92}
{\sc N.~V. Ilyushechkin}: {\em The discriminant of the characteristic
  polynomial of a normal matrix}, Mat. Zametki {\bf 51} (1992) 16--23.

\bibitem{lax98}
{\sc P.~D. Lax}: {\em On the discriminant of real symmetric matrices}, Communications
on Pure and Applied Mathematics {\bf 51}   (1998) 1387--1396.

\bibitem{muir}
{\sc T.~Muir}: {\em A Treatise on the Theory of Determinants}, Revised and
  enlarged by William H. Metzler, Dover Publications Inc., New York, 1960.
  
\bibitem{Nag}
{\sc M. Nagata}: {\em On the purity of branch loci in regular local rings},
Illinois J. Math. {\bf 3} (1959) 328--333.  
  
\bibitem{newell}
{\sc M.~J.~Newell}:
{\em On identities associated with a discriminant},
{Proc. Edinburgh Math. Soc.} {\bf 18} (1972/73) 287--291.

\bibitem{oxley}
{\sc J.~Oxley}: {\em Matroid Theory}, Oxford University Press, 1992.

\bibitem{PS}
{\sc N.~Proudfoot and D.~Speyer}: {\em A broken circuit ring}, Beitr\"age zur
  Algebra und Geometrie {\bf 47} (2006) 161--166.

\bibitem{san11}
{\sc R.~Sanyal}: {\em On the derivative cones of polyhedral cones}, Advances in
Geometry,
to appear,
{\tt arXiv:1105.2924}.

\bibitem{SchT}
{\sc H.~Schenck and {\c{S}}.~Toh{\v{a}}neanu}: {\em The {O}rlik-{T}erao algebra and 2-formality},
Math. Res. Lett. {\bf 16} (2009) 171--182.

\bibitem{Sot}
{\sc F.~Sottile}: {\em Real Solutions to Equations from Geometry},
University Lecture Series, American Mathematical Society, 
 Providence, Rhode Island, 2011.

\bibitem{Sta}
{\sc R.~P. Stanley}: {\em An introduction to hyperplane arrangements}, in
  Geometric Combinatorics, vol.~13 of IAS/Park City Math. Ser., American Mathematical
  Society, Providence, Rhode Island, 2007, pp.~389--496.

\bibitem{stu02}
{\sc B.~Sturmfels}: {\em Solving Systems of Polynomial Equations}, vol.~97 of
  CBMS Regional Conference Series in Mathematics, 
  American Mathematical Society, Providence, Rhode Island, 2002.

\bibitem{TT}
{\sc C.~Tatakis and A.~Thoma}: {\em On the universal Gr\"obner bases of toric
  ideals of graphs},
Journal of Combinatorial Theory, Series A {\bf 118} (2011) 1540--1548.

\bibitem{varchenko}
{\sc A.~Varchenko}: {\em Critical points of the product of powers of linear
  functions and families of bases of singular vectors}, Compositio Mathematica
  {\bf  97}   (1995) 385--401.

\bibitem{white}
{\sc N.~White}, ed., {\em Combinatorial Geometries}, vol.~29 of Encyclopedia
of Mathematics and its Applications, Cambridge University Press, Cambridge,
1987.

\bibitem{yy}
{\sc J.~Yu and D.~Yuster}: {\em Representing tropical linear spaces by
  circuits}, in Proceedings of FPSAC, 2007.

\bibitem{zaslavsky82b}
{\sc T.~Zaslavsky}: {\em Chromatic invariants of signed graphs}, Discrete
  Mathematics {\bf 42} (1982) 287--312.

\bibitem{zaslavsky82}
{\sc T.~Zaslavsky}: {\em Signed graphs},
  Discrete Applied Mathematics {\bf 4} (1982) 47--74.
\end{thebibliography}
%

\end{document}